\documentclass{amsart}
\usepackage{amsmath, amssymb}

\input xy
\xyoption{all}

\newtheorem{theorem}{Theorem}[section]
\newtheorem{lemma}[theorem]{Lemma}
\newtheorem{corollary}[theorem]{Corollary}
\newtheorem{fact}[theorem]{Fact}
\newtheorem{proposition}[theorem]{Proposition}

\theoremstyle{definition}
\newtheorem{example}[theorem]{Example}
\newtheorem{remark}[theorem]{Remark}
\newtheorem{definition}[theorem]{Definition}
\newtheorem{notation}[theorem]{Notation}

\def\spec{\operatorname{Spec}}

\def\id{\operatorname{id}}
\def\hom{\operatorname{Hom}}
\def\r{\operatorname{R}}

\def\O{\mathcal{O}}
\def\I{\mathcal I}
\def\E{\mathcal E}
\def\cE{\mathcal E}
\def\B{\mathcal B}
\def\F{\mathcal F}
\def\cA{\operatorname{Arc}}
\def\fm{\mathfrak m}
\def\scoj{\mathit{coJ}}
\def\jet{\operatorname{Jet}}
\def\mod{\operatorname{mod}}

\def\alg{\operatorname{alg}}

\def\im{\operatorname{Im}}

\begin{document}

\title{Jet and Prolongation spaces}

\author{Rahim Moosa}
\address{
University of Waterloo\\
Department of Pure Mathematics\\
200 University Avenue West\\
Waterloo, Ontario \  N2L 3G1\\
Canada}
\email{rmoosa@math.uwaterloo.ca}

\thanks{R. Moosa was supported by an NSERC Discovery Grant.}

\author{Thomas Scanlon}
\address{
University of California, Berkeley\\
Department of Mathematics\\
Evans Hall\\
Berkeley, CA \ 94720-3480\\
USA}
\email{scanlon@math.berkeley.edu}

\thanks{T. Scanlon was partially supported by NSF Grants DMS-0301771 and CAREER
DMS-0450010, an Alfred P. Sloan Fellowship, and a Templeton Infinity Grant.}

\date{June 25, 2008}


\begin{abstract}
The notion of a prolongation of an algebraic variety is developed in an abstract setting that generalises the difference and (Hasse) differential contexts.
An {\em interpolating} map that compares these prolongation spaces with algebraic jet spaces is introduced and studied.
\end{abstract}

\maketitle
\bigskip
\section{Introduction}

To a differentiable manifold $M$ one may associate the tangent bundle $TM$ which is itself a differentiable manifold whose points encode the information of a point in $M$ together with a tangent direction.  Through the Zariski tangent space construction, there is a natural extension of the notion of a tangent bundle to algebraic varieties and more generally to schemes.  This algebraic version of the tangent bundle admits several different interpretations in terms of derivations, dual number valued points, maps on the contangent sheaf, infinitesimal neighborhoods, \emph{et cetera}.  In  generalising the tangent bundle construction to produce spaces adapted to higher differential structure, the various aspects of the tangent bundle diverge and one may study jet spaces (in the sense of differential geometry), higher order infinitesimal neighborhoods (in the sense of Grothendieck), sheaves of differential operators, and arc spaces amongst other possibilities.  

Jet and arc spaces and their ilk appear in difference and differential algebra as prolongation spaces used to algebraize difference and differential varieties.  For example, if $(K,\partial)$ is a differential field and $X$ is an algebraic variety defined over the $\partial$-constants of $K$, then for any $K$-point $a \in X(K)$, relative to the usual presentation of the Zariski tangent bundle, we have 
$(a,\partial(a)) \in TX(K)$.  Generalizing these considerations to higher order differential operators, one may understand algebraic differential equations in terms of algebraic subvarieties of the arc spaces of algebraic varieties.  The constructions employed in difference and differential algebra bear many formal analogies and may be understood as instances of a general theory of geometry over rings with distinguished operators.  

In this paper we lay the groundwork for a careful study of several of these constructions of these spaces encoding higher order differential structure.  For us, the main goal is to develop a robust theory of linearization for generalised difference and differential equations, but to achieve this end we must study the properties of jet, arc and prolongation spaces in general.  

This paper is organized as follows.  We begin by discussing the Weil restriction of scalars construction.  This construction is, of course, well-known but we were unable to find a sufficiently detailed account in the literature.  We then introduce our formalism of $\cE$-rings simultaneously generalising difference and differential rings.  With these algebraic preliminaries in place, we define prolongation spaces of algebraic varieties over $\cE$-rings and study the geometric properties of prolongation spaces. We then switch gears to study the construction of algebraic jet spaces, which, for us, are not the same as the jet spaces considered in differential geometry.  The jet spaces of differential geometry are essentially our arc spaces, while our jet spaces are the linear spaces associated to the sheaves of higher order differential operators.  Finally, we introduce a functorial map comparing the jet space of a prolongation space with the prolongation space of a jet space and then study this interpolation map.  We conclude this article by showing that over smooth points, this interpolation map is surjective.

In the sequel, we shall apply the geometric theory developed here to build a general theory of $\cE$-algebraic geometry generalising Kolchin's theory of differential algebra and Cohn's theory of difference algebra, and in analogy with Buium's theories of arithmetic differential algebraic geometry.   In particular, the final surjectivity theorem of the present article will be used to show that for so-called separable $\cE$-varieties, $\cE$-jet spaces determine the $\cE$-varieties, thus generalising the principal results of Pillay and Ziegler~\cite{pillayziegler03} on finite-rank difference and differential varieties.

We would like to thank Sergei Starchenko for reading and commenting upon an early draft of this paper.

\smallskip

{\bf Some conventions:}
All our rings are commutative and unitary and all our ring homomorphisms preserve the identity.
All our schemes are separated.
If $X$ is a scheme over a ring $A$ and we wish to emphasise the parameters we may write $X_A$ for $X$.
Similarly, for $R$ an $A$-algebra, we may write $X_R:=X\times_AR$.
By $X(R)$ we mean the set of $R$-points of $X$ over $A$.
Note that if $k$ is an $A$-algebra and $R$ a $k$-algebra, then there is a canonical identification of $X_k(R)$ with $X(R)$.

\bigskip

\section{Weil restriction of scalars}

\noindent
Weil restriction of scalars is a notion that will be of fundamental importance in this paper.
We review some of the basic facts of this construction.
The material covered here, as well as further details, can be found in section~7.6 of~\cite{neron} and in the appendix to~\cite{oesterle84}.

Let $S$ be a scheme and $T\to S$ a scheme over $S$.
Given a scheme $Y$ over $T$, let $\mathcal{R}_{T/S}Y:\operatorname{Scheme}/S\to\operatorname{Sets}$ be the functor which assigns to any scheme $U$ over $S$ the set $\hom_T(U\times_ST,Y)$.

\begin{proposition}[cf. Theorem~4 of~\cite{neron}]
\label{weil}
Suppose $k$ is a ring, $R$ is a $k$-algebra that is finite and free as a $k$-module, $S=\spec(k)$, and $T=\spec(R)$.
If the induced morphism $T\to S$ is one-to-one then there exists a covariant functor
$$\operatorname{R}_{T/S}:\operatorname{Schemes}/T\to\operatorname{Schemes}/S$$
such that for any scheme $Y$ over $T$, $\r_{T/S}Y$ represents the functor $\mathcal{R}_{T/S}Y$.
The equivalence of the functors $\hom_S(\_,\r_{T/S}Y)$ and $\mathcal{R}_{T/S}Y$ is given by the existence of a $T$-morphism $r_Y:\r_{T/S}Y\times_ST\to Y$ such that for any scheme $U$ over $S$, $p\mapsto r_Y\circ(p\times_S\id_T)$ defines a bijection between $\hom_S(U,\r_{T/S}Y)$ and $\mathcal{R}_{T/S}Y(U)$.

We can drop the assumption that $T\to S$ is one-to-one if we restrict $\operatorname{R}_{T/S}$ to the category of schemes over $R$ that have the property that every finite set of (topological) points is contained in an open affine subscheme.
\end{proposition}

\begin{proof}
The proposition is implicit in the proof of Theorem~4 of~\cite{neron}. For the sake of completeness we provide some details here.

Choosing a basis, let $R=\bigoplus_{j=1}^tk\cdot e_j$.
We first define the functor $\r_{T/S}$ on affine schemes over $T$.
Suppose $Y=\spec\big(R[{\bf y}]/I\big)$, where ${\bf y}$ is a (possibly infinite) tuple of indeterminates.
Let $\bar{\bf y}=({\bf y}_j)_{1\leq j\leq t}$ be $t$ copies of ${\bf y}$.
Then let $\r_{T/S}Y:=\spec\big(k[\bar {\bf y}]/I'\big)$ where $I'$ is obtained as follows:
Let $\rho:R[{\bf y}]\to k[\bar {\bf y}]\otimes_kR$ be the $R$-algebra map given by $\displaystyle \rho(y)=\sum_{j=1}^ty_j\otimes e_j$ for all $y\in{\bf y}$, and $\rho(r)=1\otimes r$ for all $r\in R$.
Given $P\in R[{\bf y}]$, $\displaystyle \rho(P)=\sum_{j=1}^tP_j\otimes e_j$
where $P_1,\dots,P_t\in k[\bar {\bf y}]$ are such that
\begin{eqnarray}
\label{i'}
P(\sum_{j=1}^t{\bf y}_je_j)
& = & P_1e_1+P_2e_2+\cdots+P_te_t
\end{eqnarray}
in $\bigoplus_{j=1}^t k[\bar {\bf y}]\cdot e_j$.
If for each $1\leq j\leq t$ we let
$\pi_j:k[\bar {\bf y}]\otimes_kR
\to k[\bar {\bf y}]$
 be the map which first identifies $k[\bar {\bf y}]\otimes_kR$ with $\bigoplus_{j=1}^t k[\bar {\bf y}]\cdot e_j$ and then projects onto the $e_j$ factor, then $I'$ is the ideal generated by $\displaystyle \sum_{j=1}^t\pi_j\rho(I)$.
That is, $I'$ is the ideal generated by the $P_j$'s in~(\ref{i'}) as $P$ ranges over all polynomials in $I$.

Note that $\rho$ induces an $R$-algebra map $r_Y^*:R[{\bf y}]/I\to k[\bar {\bf y}]/I'\otimes_kR$ which in turn induces an $R$-morphism $\spec\big(k[\bar {\bf y}]/I'\big)\times_kR\to\spec\big(R[{\bf y}]/I\big)$.
This is the $T$-morphism $r_Y:\r_{T/S}Y\times_ST\to Y$ whose existence is asserted in the Proposition.

We need to define $\r_{T/S}$ on morphisms.
Suppose $p:Y\to Z$ is an $R$-morphism, where $Z=\spec\big(R[{\bf z}]/J\big)$.
Then $\r_{T/S}(p):\spec\big(k[\bar {\bf y}]/I'\big)\to\spec\big(k[\bar {\bf z}]/J'\big)$ is the map induced by the $k$-algebra map $z_j\mapsto\pi_jr_Y^*(p^*z)$, where $p^*:R[{\bf z}]/J\to R[{\bf y}]/I$ is the $R$-map on co-ordinate rings associated to $p$ and $r_Y^*$ and $\pi_j$ are as in the preceeding paragraphs.
It is routine to check that $\r_{T/S}$ thus defined is indeed a functor from affine schemes over $T$ to affine schemes over $S$.

Next we show that $\r_{T/S}Y$ does indeed represent the functor $\mathcal{R}_{T/S}Y$ (still restricting to affine schemes).
Suppose $U=\spec(A)$ is an affine scheme over $k$.
We first show that $p\mapsto r_Y\circ(p\times_k\id_R)$ defines a bijection between $\hom_k(U,\r_{T/S}Y)$ and $\mathcal{R}_{T/S}Y(U)=\hom_R(U\times_kR,Y)$.
Working with the co-ordinate rings instead, we need to show that
$f\mapsto (f\times\id_R)\circ r_Y^*$ gives a bijection from $\hom_k(k[\bar {\bf y}]/I',A)$ to $\hom_R(R[{\bf y}]/I,A\otimes_kR)$.
For injectivity we just observe that if $(f\times\id_R)\circ r_Y^*=(g\times\id_R)\circ r_Y^*$ then $\displaystyle \sum_{j=1}^tf(y_j+I')\otimes e_j=\sum_{j=1}^tg(y_j+I')\otimes e_j$ for all $y\in{\bf y}$.
But since $A\otimes_kR=\bigoplus_{j=1}^tA\cdot e_j$, this implies that $f(y_j+I')=g(y_j+I')$ for all $y\in{\bf y}$ and all $1\leq j\leq t$.
So $f$ and $g$ agree on the generators of $k[\bar {\bf y}]/I'$ and hence are equal.
For surjectivity, suppose $\alpha\in\hom_R(R[{\bf y}]/I,A\otimes_kR)$.
For each $y\in{\bf y}$ write $\displaystyle \alpha(y+I)=\sum_{j=1}^ta_j\otimes e_j$, where $a_j\in A$.
Now define $f:k[\bar {\bf y}]/I'\to A$ by $f(y_j+I'):=a_j$.
Then $f\in \hom_k(k[\bar {\bf y}]/I',A)$ and we compute that
$$\big[(f\times\id_R)\circ r_Y^*\big] (y+I) =
(f\times\id_R)\big(\sum_{j=1}^t(y_j+I')\otimes e_j\big) =
\sum_{j=1}^ta_j\otimes e_j =
\alpha(y+I)$$
 for each $y\in{\bf y}$, as desired.

It is routine to check that the above bijection is functorial in $U$ and therefore does establish the desired equivalence of functors, at least restricted to affine schemes.

It remains therefore only to go from affine schemes to schemes in general.
In Theorem~4 of~\cite{neron} there is an argument going from the representability of $\mathcal{R}_{T/S}Y$ for affine schemes $Y$ to the representability of $\mathcal{R}_{T/S}Z$ where $Z$ has the property that every finite set of (topological) points is contained in some open affine subscheme.
As a matter of fact of the proof given there, one only has to worry about finite sets in $Z$ of cardinality bounded by the cardinality of the fibres of $T\to S$ (they work with the weaker assumption that $T\to S$ is finite and locally free).
Hence, in the case that $T\to S$ is one-to-one,
the extra assumption on $Z$ is unnecessary.

So $\r_{T/S}$ extends to a functor on all schemes over $T$.
Finally, from the fact that the morphism $r_Y$ defined above for affine schemes is functorial in $Y$, it is not hard to check that it extends to all schemes over $T$ with the desired property.
\end{proof}

\begin{definition}
\label{weilrestriction}
Suppose $k$ is a ring, $A$ is a $k$-algebra that is finite and free over $k$, and $Y$ is a scheme over $A$ such that either $\spec(A)\to\spec(k)$ is a homeomorphism or $Y$ has the property that every finite set of points is contained in an affine open subset.
Then the scheme $\r_{\spec(A)/\spec(k)}Y$ given by the above proposition is called the {\em Weil restriction} of $Y$ from $A$ to $k$ and will be denoted by $\r_{A/k}Y$.
\end{definition}

The following fact, which is stated in a seemingly weaker but in fact equivalent form in section 7.6 of~\cite{neron}, is a routine diagram chase.

\begin{fact}
\label{weilbc}
Weil restriction is compatible with base change.
That is, if $T\to S$ and $S'\to S$ are schemes over $S$, and $T':=T\times_S S'$, and $\mathcal{R}_{T/S}$. Then for any scheme $Y$ over $T$,
$$\r_{T/S}Y\times_SS'=\r_{T'/S'}(Y\times_TT')$$
whenever $\r_{T/S}Y$ and $\r_{T'/S'}(Y\times_TT')$ exist.
\end{fact}

\begin{fact}
\label{weiletale}
If $T \to S$ is a scheme over $S$ and $f:X \to Y$ is a smooth (respectively \'etale) morphism of schemes over $T$, then $\r_{T/S}(f):\r_{T/S}X \to \r_{T/S}Y$ is also smooth (respectively \'etale) -- whenever $\r_{T/S}X$ and $\r_{T/S}Y$ exist.
\end{fact}
\begin{proof}
This proof is essentially the same as the argument for part (h) of Proposition~5 in section~7.6 of~\cite{neron}.  
We spell out some of the details. 

We use the characterisation of smooth and \'etale maps given by Proposition~6 of section~2.2 of~\cite{neron}; namely that it suffices to show that for any affine scheme $Z\to\r_{T/S}(Y)$ and all closed subschemes $Z_0$ of $Z$ whose ideal sheaf is square zero, the canonical map
$\hom_{\r_{T/S}Y}\left(Z,\r_{T/S}X\right)\to\hom_{\r_{T/S}Y}\left(Z_0,\r_{T/S}X\right)$ is surjective (respectively bijective).
So given $a:Z_0\to \r_{T/S}X$ over $\r_{T/S}Y$ we want to lift it (uniquely) to $Z$.
Base changing up to $T$ we obtain
$$\xymatrix{
Z_0\times_ST\ar[dr]\ar[r]^{a\times\id\quad\quad} & \r_{T/S}X\times_ST\ar[d]^{\r_{T/S}(f)\times\id}\ar[r]^{\quad\quad r_X} & X\ar[d]^f\\
& \r_{T/S}Y\times_ST\ar[r]^{\quad \quad r_Y} & Y
}$$
Since $f$ is smooth $r_X\circ(a\times\id)$ lifts to a map $p:Z\times_ST\to X$ over $Y$.
Now, under the identification of $X(Z\times_ST)$ with $\r_{T/S}X(Z)$, $p$ corresponds to a map $\tilde{p}:Z\to\r_{T/S}X$  such that $p=r_X\circ(\tilde{p}\times\id)$.
So we get
$$\xymatrix{
Z\times_ST\ar[dr]\ar[r]^{\tilde{p}\times\id\quad\quad} & \r_{T/S}X\times_ST\ar[d]^{\r_{T/S}(f)\times\id}\ar[r]^{\quad\quad r_X} & X\ar[d]^f\\
& \r_{T/S}Y\times_ST\ar[r]^{\quad\quad r_Y} & Y
}$$
Since $p=r_X\circ(\tilde{p}\times\id)$ extends $r_X\circ(a\times\id)$, $\tilde{p}$ is our desired extension of $a$.
Moreover, if $\tilde{p}'$ were another lifting of $a$ to $Z$ then $r_x\circ(\tilde{p}'\times\id)$ would be another lifting of $r_X\circ(a\times\id)$.
Hence, if $f$ is \'etale then $\tilde{p}$ is the unique extension of $a$ to $Z$.
\end{proof}      
 
\begin{example}[Arc spaces]
\label{example-arc}
Let $X$ be a scheme over a ring $k$.
Letting $k^{(n)}:=k[\epsilon]/(\epsilon^{n+1})$, we define the {\em $n^\text{th}$ arc space $\cA_n(X)$ of $X$}, to be the Weil restriction of $X\times_kk^{(n)}$ from $k^{(n)}$ to $k$.
(Note that $\spec(k^{(n)})\to\spec(k)$ is a homeomorphism so that the Weil restriction does exist.)
In particular, $\cA_n(X)(k)$ is canonically identified with the $k^{(n)}$-points of $X\times_k k^{(n)}$, which in turn can be identified with $X(k^{(n)})$.
Note that $\cA_1(X)$ is the tangent bundle $TX$.

The quotient $k^{(n)}\to k$ induces a structure map $\cA_n(X)\to X$ as follows:
For any $k$-algebra $R$,
$\cA_n(X)(R)$ is identified with $(X\times_kk^{(n)})(R\otimes_kk^{(n)})$ which is in turn identified with $X(R\otimes_kk^{(n)})$.
Now $R\otimes_kk^{(n)}=R[\epsilon]/(\epsilon^{(n+1)})=:R^{(n)}$, and we have the natural quotient map $\rho_n^R:R^{(n)}\to R$ inducing $\underline{\rho_n^R}:\spec(R)\to \spec(R^{(n)})$.
Composition with $\underline{\rho_n^R}$ defines a map from the $X(R^{(n)})$ to $X(R)$.
That is, from $\cA_n(X)(R)$ to $X(R)$.

\end{example}

\bigskip

\section{$\mathcal{E}$-rings}
\label{sectering}

\begin{definition}
\label{ringscheme}
By the {\em standard ring scheme} $\mathbb{S}$ over $A$ we mean the ring scheme where $\mathbb{S}(R)=(R,+,\times,0,1)$, for all $A$-algebras $R$.
An {\em $\mathbb{S}$-algebra scheme} $\E$ over $A$ is a ring scheme together with a ring scheme morphism $s_\E:\mathbb{S}\to\E$ over $A$.
We view $\mathbb{S}$ as an $\mathbb{S}$-algebra via the identity $\id:\mathbb{S}\to\mathbb{S}$.
A {\em morphism} of $\mathbb{S}$-algebra schemes is then a morphism of ring schemes respecting the $\mathbb{S}$-algebra structure.
Similarly one can define {\em $\mathbb{S}$-module} schemes and morphisms.
By a {\em finite free} $\mathbb{S}$-algebra scheme we will mean, somewhat unnaturally, an $\mathbb{S}$-algebra scheme $\mathcal{E}$ together with an isomorphism of $\mathbb{S}$-module schemes $\psi_\E:\mathcal{E}\to\mathbb{S}^\ell$, for some $\ell\in\mathbb{N}$.
\end{definition}

The intention behind fixing the ismorphism $\psi_\E:\mathcal{E}\to\mathbb{S}^\ell$ is to give us a canonical way of presenting $\mathcal{E}(R)$ uniformly in all $A$-algebras $R$.
Indeed, we get $\E(R)=R[(e_i)_{i\leq\ell}]/I_R$, where the $e_i$ are indeterminates and $I_R$ is the ideal generated by expressions that describe how the monomials in the $e_i$'s can be written as $A$-linear combinations of the $e_i$'s.

\begin{remark}
\label{presentation}
It follows from the above discussion that for any $A$-algebra, $\alpha:A\to R$, we can canonically identify $\mathcal{E}(R)$ and $\mathcal{E}(A)\otimes_AR$, both as $R$-algebras and as $\E(A)$-algebras (where the $\E(A)$-algebra structure on $\E(R)$ is by $\E(\alpha):\E(A)\to\E(R)$).
The converse is also true: given any finite and free $A$-algebra $B$, by choosing a basis we can find a finite and free $\mathbb{S}$-algebra scheme $\E$ over $A$ such that $\E(R)=B\otimes_AR$ for any $A$-algebra $R$.
\end{remark}

\begin{definition}
\label{ering}
Suppose $\mathcal{E}$ is a finite free $\mathbb{S}$-algebra scheme over a ring $A$.
An {\em $\mathcal{E}$-ring} is an $A$-algebra $k$ together with an $A$-algebra homomorphism $e:k\to\mathcal{E}(k)$.
\end{definition}

One may equally well describe an $\E$-ring by giving a collection of operators 
$\{ \partial_i:k \to k \}_{i \leq \ell}$ via the correspondence 
$\psi_{\E} \circ e = (\partial_1,\ldots,\partial_\ell)$.  That the 
collection $\{ \partial_i \}$ so defines an $\E$-ring structure on $k$ is equivalent to 
the satisfaction of a certain system of functional equations.

We now discuss some examples.

\begin{example}[Pure rings]
\label{pure}
For any $A$-algebra $k$, $(k,\id)$ is an $\mathbb{S}$-ring.
\end{example}

\begin{example}[Rings with endomorphisms]
\label{difference}
Fix $n\geq 0$ and consider the finite free $\mathbb{S}$-algebra scheme $\E_n=\mathbb{S}^n$ with the product ring scheme structure and $s:\mathbb{S}\to\mathbb{S}^n$ being the diagonal.
If $k$ is a ring and  $\sigma_0,\dots,\sigma_{n-1}$ are endomorphisms of $k$, then $(k,e_n)$ is an $\E_n$-ring where $e_n:=(\sigma_0,\sigma_1,\dots,\sigma_{n-1})$.
In particular, if $\sigma$ is an automorphism of $k$, then setting $\sigma_0=\id$, and setting $\sigma_{2m-1}=\sigma^m$ and $\sigma_{2m}=\sigma^{-m}$ for all $m>0$, we see that the difference ring structure is captured by $\big\{(k,e_n):n\in\mathbb{N}\big\}$.
\end{example}

\begin{example}[Hasse-differential rings]
\label{differential}
For each $n\geq 0$, consider the finite free $\mathbb{S}$-algebra scheme $\E_n$ where for any ring $R$
\begin{itemize}
\item
$\E_n(R)=R[\eta_1,\dots,\eta_r]/(\eta_1,\dots,\eta_r)^{n+1}$, where $\eta_1,\dots,\eta_r$ are indeterminates;
\item
$s^R:R\to\E_n(R)$ is the natural inclusion; and,
\item
$\psi^R: \E_n(R)\to R^{\ell_n}$ is the identification via the standard monomial basis of $R[\eta_1,\dots,\eta_r]/(\eta_1,\dots,\eta_r)^{n+1}$ over $R$;
\end{itemize}
We leave it to the reader to write down the equations which verify that such a finite free $\mathbb{S}$-algebra scheme exists.

Recall that a {\em Hasse derivation} on a ring $k$ is a sequence of additive maps from $k$ to $k$, ${\bf D}=(D_0,D_1\dots)$, such that
\begin{itemize}
\item
$D_0=\id$ and
\item
$\displaystyle D_m(xy)=\sum_{a+b=m}D_a(x)D_b(y)$ for all $m>0$.
\end{itemize}
Suppose ${\bf D}_1,\dots,{\bf D}_r$ is a {\em sequence} of $r$ Hasse derivation on $k$ and set $\displaystyle E(x)=\sum_{\alpha\in\mathbb{N}^r} D_{1,\alpha_1}D_{2,\alpha_2}\cdots D_{r,\alpha_r}(x)\eta^{\alpha}$.
Then $E:k\to k[[\eta_1,\dots,\eta_r]]$ is a ring homomorphism.
Let $e_n$ be the composition of $E$ with the quotient $k[[\eta_1,\dots,\eta_r]]\to k[\eta_1,\dots,\eta_r]/(\eta_1,\dots,\eta_r)^{n+1}$.
Then $(k,e_n)$ is an $\E_n$-ring.
The Hasse-differential ring structure is captured by the sequence $\big\{(k,e_n):n\in\mathbb{N}\big\}$.

This example specialises to the case of {\em partial differential fields in characteristic zero}.
Suppose $k$ a field of characteristic zero and $\partial_1,\dots,\partial_r$ are derivations on $k$.
Then $\displaystyle D_{i,n}:=\frac{\partial_i^n}{n !}$, for $1\leq i\leq r$ and $n\geq 0$, defines a sequence of Hasse derivations on $k$.
The $\E_n$-ring structure on $k$ is given in multi-index notation by $\displaystyle e_n(x):=\sum_{\alpha\in\mathbb{N}^r, |\alpha|\leq n}\frac{1}{\alpha!}\partial^\alpha(x)\eta^\alpha$ where $\partial:=(\partial_1,\dots,\partial_r)$.

On the other hand we can specialise in a different direction to deal with {\em fields of finite imperfection degree}.
The following example is informed by~\cite{ziegler03}:
suppose $k$ is a field of characteristic $p>0$ with imperfection degree $r$.
Let $t_1,\dots, t_r$ be a $p$-basis for $k$.
Then $t_1,\dots, t_r$ are algebraically independent over $\mathbb{F}_p$.
Consider $\mathbb{F}_p[t_1,\dots,t_r]$ and for $1\leq i\leq r$ and $n\in\mathbb{N}$, define
$${\bf D}_{i,n}(t_1^{\alpha_1}\cdots t_r^{\alpha_r}):=\left(\begin{array}{c}\alpha_i\\n\end{array}\right)t_1^{\alpha_1}\cdots t_i^{\alpha_i-n}\cdots t_r^{\alpha_r}.$$
and extend by linearity to $\mathbb{F}_p[t_1,\dots,t_r]$.
The sequence $({\bf D}_1,\dots,{\bf D}_r)$ forms a sequence of $r$ Hasse derivations on this domain.
Moreover, they extend uniquely to Hasse derivations on $k$ (see Lemma~2.3 of~\cite{ziegler03}).
This gives rise to a $\E_n$-ring structure on $k$ for all $n$.
\end{example}

The above examples can also be combined to treat difference-differential rings.

Next we consider an example interpolating between differential and difference rings.
The notion of {\em $D$-ring} was introduced by the second author in~\cite{scanlon2000}.

\begin{example}[$D$-rings]
Let $A$ be a commutative ring having a distinguished element $c \in A$.  For any $A$-algebra $R$ we define ${\mathcal E}_c(R)$ to be the $A$-module $R \times R$ with multiplication defined by $(x_1,x_2) \cdot (y_1,y_2) := (x_1 y_1, x_1 y_2 + x_2 y_1 + c x_2 y_2)$. 
To give $R$ the structure of a ${\mathcal E}_c$-ring we need only produce an $A$-linear map $D:R \to R$ for which $R \to {\mathcal E}_c(R)$ given by $x \mapsto (x,D(x))$ is a homomorphism of $A$-algebras.
Note that if $c = 0$, then $\E_c(R)$ is the ring of dual numbers over $R$ and $D$ is a derivation.
At the other extreme, if $c \in A^\times$, then from any $A$-algebra endomorphism $\tau:R \to R$, we give $R$ a $\E_c$-ring structure via $D(x) := c^{-1}[\tau(x) - x]$.
Considered at the level of the ring schemes, we see that $(\E_c)_{A[\frac{1}{c}]} \approx {\mathbb S}_{A[\frac{1}{c}]} \times {\mathbb S}_{A[\frac{1}{c}]}$ while $(\E_c)_{A/(c)} \approx ({\mathbb S}[\epsilon]/(\epsilon^2))_{A/(c)}$.
In particular, when $A$ is a field, an $\E_c$-ring is essentially either a difference ring or a differential ring.
\end{example}

Finally let us discuss an example of a ring functor which does not fit into our formalism, but for which some of our constructions still make sense.

\begin{example}[$\lambda$-rings]
Fix $A$ a commutative ring of characteristic $p$ with a finite $p$-independent set $B$.
For each natural number $n$ and $A$-algebra $R$, define $\E_n(R) := R[\{ s_b : b \in B \} ]/( \{ s_b^{p^n} - b : b \in B \}]$ and take for its basis over $A$ the monomials in $\{ s_b : b \in B \}$ in which each variable appears to degree less than $p^n$.
As it stands, an $\E_n$-ring $(A,\psi)$ is simply a Hasse differential ring for which certain linear differential operators must vanish identically.
However, we might consider a variation on the definition of an $\E$-ring:
Letting $P_n:\E_n(R) \to R$ be given by $x \mapsto x^{p^n}$, we might ask for a map $\lambda:R \to \E_n(R)$ for which the composite $P_n \circ \lambda$ is the identity on $R$.
(Note that $\lambda$ is then not $A$-linear, and hence does not give an $\E_n$-ring structure in our sense.)
With this construction we recover the formalism of $\lambda$-functions used in the study of the theory of separably closed fields.
\end{example}

As the above examples suggest, one is usually interested in a whole sequence of $\E_n$-ring structures on $k$ that satisfy certain compatibility conditions.
A more systematic study of such systems, {\em generalised Hasse systems}, with their attendant geometry, will be carried out in the sequel to this paper.

We conclude this section with some notation.

\begin{notation}
\label{2oner}
Suppose $(k,e)$ is an $\mathcal{E}$-ring.
\begin{itemize}
\item[(a)]
Note that $\E(k)$ has two $k$-algebra structures, the {\em standard} $s_\E^k:k\to \E(k)$ and the {\em exponential} $e:k\to\E(k)$.
Both induce the same $A$-algebra structure.
In order to distinguish these notationally, we will denote the latter by $\E^e(k)$.
\item[(b)]
The ring homomorphism $e:k\to\E(k)$ also induces a second $k$-algebra structure on $\E(R)$, for any $k$-algebra $R$.
Namely, given $a:k\to R$ we obtain
$$\xymatrix{
k\ar[rr]^e &&\E(k)\ar[rr]^{\E(a)} && \E(R)
}$$
We denote this $k$-algebra by $\E^e(R)$.
Alternatively, if we identify $\E(R)$ with $R\otimes_k\E(k)$ as in Remark~\ref{presentation}, then $\E^e(R)$ is described by
$$\xymatrix{
k\ar[rr]^e &&\E(k)\ar[rr] && R\otimes_k\E(k)
}$$
Note that in general $\E^e(R)\neq R\otimes_k \E^e(k)$ as $k$-algebras.
\end{itemize}
\end{notation}

\bigskip
\section{Abstract Prolongations}
\label{sectprolong}
\noindent
We fix a finite free $\mathbb{S}$-algebra scheme $\mathcal{E}$ over a ring $A$ and an $\mathcal{E}$-ring $(k,e)$.
The following definition is partly informed by, and generalises, a construction of Buium's in the case of an ordinary differential ring (cf. $9.1$ of~\cite{buium92}).\footnote{Buium calls his construction a differential ``jet space'', which conflicts badly with our terminology in several ways.}

\begin{definition}[Prolongation]
\label{prolongation}
Suppose $X$ is a scheme over $k$.
The {\em prolongation space of $X$ with respect to $\E$ and $e$}, denoted by $\tau(X,\E,e)$, is the Weil restriction of $X\times_k\E^e(k)$ from $\mathcal{E}(k)$ to $k$, when it exists.
Note that we are taking the base extension with respect to the exponential $e:k\to\mathcal{E}(k)$, while we are taking the Weil restriction with respect to the standard $s_\E^k:k\to \E(k)$.
When the context is clear we may write $\tau(X)$ for $\tau(X,\E,e)$.
\end{definition}

So for any $k$-algebra $R$, using the fact (Remark~\ref{presentation}) that $\E(R)=\E(k)\otimes_kR$, we have that
$\displaystyle \tau(X)(R)=\big(X\times_k\E^e(k)\big)\big(\E(R)\big)$.

Note that if $X$ is quasi-projective then the prolongation space necessarily exists regardless of the $\E$-ring -- this is because every finite set of points in a quasi-projective scheme is contained in an affine open subset, and that is the condition for the existence of the Weil restriction.
On the other hand, for particular $\E$-rings, the prolongation spaces may exist for other schemes -- for example if $\underline{s^k}:\spec\big(\E(k)\big)\to \spec(k)$ is one-to-one then $\tau(X,\E,e)$ exists for all schemes $X$ over $k$ (cf. Proposition~\ref{weil}).
{\em In the rest of this paper we implicitly assume that our schemes and $\E$-rings are such that the the relevant prolongation spaces exist.}
If the reader is uncomfortable with this sleight of hand, he/she is welcome to assume that all our schemes are quasi-projective.

\begin{example}
\begin{itemize}
\label{prolongexamples}
\item[(a)]
If $\E=\mathbb{S}$ and $e=s_\E^k=\id$, then $\tau(X)=X$.

\item[(b)]
Arc spaces are prolongations.
If $\E$ is the $\mathbb{S}$-algebra scheme given by $\E(R)=R[\epsilon]/(\epsilon^{n+1})$ and $e=s_\E^k$ is the standard $k$-algebra structure on $k[\epsilon]/(\epsilon^{n+1})$, then $\tau(X)=\cA_n(X)$ of Example~\ref{example-arc}.

\item[(c)]
Difference rings.
If $\E_n(k)=k^n$ is as in Example~\ref{difference}, and $e_n=(\sigma_0,\dots,\sigma_{n-1})$ is a sequence of endomorphisms of $k$, then $\tau(X,\E_n,e_n)= X^{\sigma_0}\times\cdots\times X^{\sigma_{n-1}}$, where $X^{\sigma_i}=X\times_{\sigma^i}k$.
Indeed, for any scheme $U$ over $k$, a $k^n$-morphism from $U\times_sk^n$ to $X\times_{e_n}k^n$ determines and is determined by a sequence of $k$-morphisms from $U$ to $X^{\sigma_i}$, $i=0,\dots,{n-1}$.

\item[(d)]
Differential rings.
Suppose $\E$ is the $\mathbb{S}$-algebra scheme given by $\E(R)=R[\eta]/(\eta^{2})$ with the standard $k$-algebra structure.
Suppose $k$ is a field of characteristic zero and $\delta$ is a derivation on $k$ and $\displaystyle e(a)=a +\delta(a)\eta$.
If $X$ is the affine scheme $\spec\big(k[x_1,\dots x_m]/\langle P_1,\dots,P_t\rangle\big)$, then the Weil restriction computations shows that $\tau(X,\E,e)$ is the affine subscheme of $\mathbb{A}^{2m}$ whose ideal is generated by $P_1(x),\dots,P_t(x)$ together with
$$\sum_{i=1}^m\frac{\partial}{\partial x_i}P_j(x)\cdot y_i \ +P_j^{\delta}(x)$$
for $j=1,\dots,t$, where $P^{\delta}$ is obtained from $P$ by applying $\delta$ to the coefficients.
So if $X$ is over the constants of $\delta$ then this prolongation space is the tangent bundle.
\end{itemize}
\end{example}

\begin{definition}
\label{canonicalmorph}
By the {\em canonical morphism associated to $\tau(X,\E,e)$}, denoted
$$r^X_{\E,e}:\tau(X)\times_k\E(k)\to X,$$
we mean the composition of the $\E(k)$-morphism
$\tau(X)\times_k\E(k)\to X\times_k\E^e(k)$ given by the representability of the Weil restriction (cf. Proposition~\ref{weil}) and the projection $X\times_k\E^e(k)\to X$.
\end{definition}

\begin{remark}
\label{rnotoverk}
The canonical morphism $r^X_{\E,e}:\tau(X)\times_k\E(k)\to X$ is {\em not} a $k$-morphism if we view $\tau(X)\times_k\E(k)$ as over $k$ in the usual way.
However we do have:
$$\xymatrix{
\tau(X)\times_k\E(k)\ar[dr]\ar[rr]^{r^X_{\E,e}}&&X\ar[dd]\\
&\spec\big(\E(k)\big)\ar[dr]_{\underline{e}} &\\
&&\spec(k)
}$$
where $\underline{e}$ is the morphism of schemes induced by $e:k\to\E(k)$.
\end{remark}

As mentioned earlier, the prolongation space is characterised by the property that, for any $k$-algebra $R$, $\tau(X)(R)=\big(X\times_k\E^e(k)\big)\big(\E(R)\big)$.
However, the following lemma gives another useful description of the $R$-points of the prolongation.

\begin{lemma}
\label{taunpoints}
For any $k$-algebra $R$, $\tau(X)(R)=X\big(\E^e(R)\big)$.
More precisely, the $R$-points of $\tau(X)$ over $k$ can be functorially identified with the $\E^e(R)$-points of $X$ over $k$.
This identification is given by $p\mapsto r^X_{\E,e}\circ\big(p\times_k\E(k)\big)$.
\end{lemma}

\begin{proof}
First of all, the defining property of the Weil restriction implies that
$$\tau(X)(R)=X\times_k\E^e(k)\big(R\otimes_k\E(k)\big)$$
where the identification is obtained by associating to
$$\xymatrix{\spec(R)\ar[r]^p&\tau(X)}$$
the $\E(k)$-morphism
$$\xymatrix{\spec\big(R\otimes_k\E(k)\big)\ar[rr]^{p\times_k\E(k)}&&\tau(X)\times_k\E(k)\ar[rr]&&X\times_k\E^e(k)}$$
where $\tau(X)\times_k\E(k)\to X\times_k\E^e(k)$ is given by the representability of the Weil restriction (cf. Proposition~\ref{weil}).
On the other hand,
$$X\times_k\E^e(k)\big(R\otimes_k\E(k)\big)=X\big(\E^e(R)\big)$$
Indeed, given
$$\xymatrix{\spec\big(R\otimes_k\E(k)\big)\ar[rr]^q&&X\times_k\E^e(k)}$$
consider the following diagram
$$\xymatrix{
& \spec\big(R\otimes_k\E(k)\big)\ar[dl]\ar[dr]\ar[rr]^{q} && X\times_k\E^e(k)\ar[dr]\ar[dl]\\
\spec(R)\ar[dr] & & \spec\big(\E(k)\big)\ar[dr]_{\underline{e}}\ar[dl]^{\underline{s_\E^k}} && X\ar[dl]\\
& \spec(k)&&\spec(k)
}$$
where $\underline{e}$ and $\underline{s_\E^k}$ are the morphisms on schemes induiced by $e$ and $s_\E^k$, respectively.
We see that composing $q$ with the projection $X\times_k\E^e(k)\to X$ gives us the natural identification of the $\big(R\otimes_k\E(k)\big)$-points of $X\times_k\E^e(k)$ over $\E(k)$ with the $\big(R\otimes_k\E(k)\big)$-points of $X$ over $k$ {\em where $R\otimes_k\E(k)$ is viewed as a $k$-algebra by}
$$\xymatrix{k\ar[r]^e&\E(k)\ar[r]&R\otimes_k\E(k)}$$
But $\E^e(R)$ is canonically isomorphic to $R\times_k\E(k)$ with the above $k$-algebra structure (cf. ~\ref{2oner}).
Hence $\tau(X)(R)=X\big(\E^e(R)\big)$, as desired.
\end{proof}

The prolongation space construction is a covariant functor:
If $f:X\to Y$ is a morphism of schemes over $k$ then $\tau(f)=\r_{\E(k)/k}\big(f\times_k\E^e(k)\big)$ is the morphism given by the Weil restriction functor applied to $f\times_k\E^e(k):X\times_k\E^e(k)\to Y\times_k\E^e(k)$.
Alternatively, $\tau(f)$ can be described on $R$-points for any $k$-algebra $R$, after identifying $\tau(X)(R)$ with $X\big(\E^e(R)\big)$ and $\tau(Y)(R)$ with $Y\big(\E^e(R)\big)$, as composition with $f$.

\begin{proposition}
\label{etaletau}
If $f:X \to Y$ is an \'etale morphism (respectively closed embedding, smooth morphism), then $\tau(f):\tau(X)\to\tau(Y)$ is \'etale (respectively a closed embedding, smooth).
In particular, if $X$ is a smooth then so is $\tau(X)$.
\end{proposition}

\begin{proof}
Weil restrictions preserve smooth morphisms, \'etale morphisms (Fact~\ref{weiletale}) and closed embeddings (the latter is clear from the construction, see the proof of Proposition~\ref{weil}).
This is also true of base change.
It therefore follows that the prolongation functor preserves all these properties.
Since $X$ being smooth is equivalent $X\to\spec(k)$ being smooth, it follows that $\tau(X)$ is smooth if $X$ is.
\end{proof}

There is a natural map $\nabla=\nabla^X_{\E,e}:X(k)\to\tau(X)(k)$
\label{nabla}
induced by $e$ as follows:
writing $e$ as a $k$-algebra homomorphism $e:k\to \E^e(k)$ we see that it  induces a map from the $X(k)$ to $X\big(\E^e(k)\big)$.
This, together with the identification $\tau(X)(k)=X\big(\E^e(k)\big)$ from Lemma~\ref{taunpoints}, gives us $\nabla:X(k)\to\tau(X)(k)$.

\begin{proposition}
\label{nablafunctorial}
Suppose $f:X\to Y$ is a morphism of schemes over $k$.
\begin{itemize}
\item[(a)]
The following diagram commutes:
$$\xymatrix{
\tau(X)(k)\ar[rr]^{\tau(f)}&&\tau(Y)(k)\\
X(k)\ar[u]^{\nabla^X}\ar[rr]^f&& Y(k)\ar[u]_{\nabla^Y}
}$$
\item[(b)]
Suppose $a\in Y(k)$.
Then $\tau(X)_{\nabla(a)}=\tau(X_a)$, where $\tau(X)_{\nabla(a)}$ is the fibre of $\tau(f):\tau(X)\to\tau(Y)$ over $\nabla(a)$.
\end{itemize}
\end{proposition}

\begin{proof}
Identifying $\tau(X)(k)$ with $X\big(k\otimes_e\E(k)\big)$ and $\tau(Y)(k)$ with $Y\big(k\otimes_e\E(k)\big)$, $\tau(f)$ and $f$ on $k$-points are both given by composing with $f$.
On the other hand, by definition, under the same identifications, $\nabla^X$ and $\nabla^Y$ are both given by pre-composition with $\underline{e}:\spec\big(k\otimes_e\E(k)\big)\to\spec(k)$.
Part~(a) follows immediately.

We check part~(b) at $R$-points for any given $k$-algebra $R$.
Under the canonical identifications,
$\nabla(a)\in\tau(Y)(k)$ is given by $\underline{e}(a)\in Y\big(\E^e(k)\big)$ as in the diagram
$$\xymatrix{
\spec\big(\E^e(k)\big)\ar[d]^{\underline{e}}\ar[dr]^{\underline{e}(a)}\\
\spec(k)\ar[r]^a & Y}
$$
Thus $\tau(X)_{\nabla(a)}(R)$ is identified with $X_{\underline{e}(a)}\big(\E^e(R)\big)$ where $X_{\underline{e}(a)}:=X\times_Y\spec(\E^e(k)\big)$.
On the other hand, under the same canonical identification, we have $\tau(X_a)(R)=X_a\big(\E^e(R)\big)$, where $X_a:=X\times_Y\spec(k)$.
But note that $X_{\underline{e}(a)}=X_a\times_k{\E^e(k)}$.
and so $X_{\underline{e}(a)}\big(\E^e(R)\big)=X_a\big(\E^e(R)\big)$.
\end{proof}

\bigskip
\subsection{Comparing prolongations}
\label{compare}
Fix two finite free $\mathbb{S}$-algebra schemes $\mathcal{E}$ and $\F$ over a ring $A$, together with a ring-scheme morphism $\alpha:\E\to\F$ over $A$.
Suppose $k$ is an $A$-algebra  and $e$ and $f$ are such that $(k,e)$ is an $\mathcal{E}$-ring, $(k,f)$ is an $\mathcal{F}$-ring, and $\alpha\circ e=f$ (so that $\alpha^k:\E^e(k)\to\F^f(k)$ is a $k$-algebra homomorphisms).
For any $k$-algebra $R$, since $\alpha^R$ lifts $\alpha^k$, it follows that $\alpha^R:\E^e(R)\to\F^f(R)$ is also a $k$-algebra homomorphism.

Given a scheme $X$ over $k$, $\alpha$ induces a morphism of schemes $\hat \alpha:\tau(X,\E,e)\to\tau(X,\F,f)$.
Indeed, for any $k$-algebra $R$, pre-composition with the induced morphism of schemes $\spec\big(\F^f(R)\big)\to\spec\big(\E^e(R)\big)$ over $k$, in turn induces a map from $X\big(\E^e(R)\big)$ to $X\big(\F^f(R)\big)$.
We now point out some of the properties of this morphism.

\begin{proposition}
\label{hatproperties}
\begin{itemize}
\item[(a)]
The following diagram commutes:
$$\xymatrix{
\tau(X,\E,e)(k)\ar[rr]^{\hat\alpha} && \tau(X,\F,f)(k)\\
& X(k)\ar[ul]^{\nabla_\E}\ar[ur]_{\nabla_\F}
}$$
\item[(b)]
Suppose $f:X\to Y$ is a morphism of schemes over $k$.
Then the following diagram commutes:
$$\xymatrix{
\tau(X,\E,e)\ar[d]_{\hat\alpha^X}\ar[rr]^{\tau^\E(f)}&&\tau(Y,\E,e)\ar[d]^{\hat\alpha^Y}\\
\tau(X,\F,f)\ar[rr]^{\tau^\F(f)}&& \tau(Y,\F,f)
}$$
\item[(c)]
If $\alpha:\E\to\F$ is a closed embedding, then so is $\hat\alpha:\tau(X,\E,e)\to\tau(X,\F,f)$.
\end{itemize}
\end{proposition}

\begin{proof}
Part~(a) is immediate from the definitions using the fact that $\alpha$ preserves the $k$-algebra structures coming from $e$ and $f$.

We show the diagram in part~(b) commutes by evaluating on $R$-points for an arbitrary $k$-algebra $R$.
Making the usual identifications, we need to show that the following diagram commutes:
$$\xymatrix{
X\big(\E^e(R)\big)\ar[d]\ar[r] & Y\big(\E^e(R)\big)\ar[d]\\
X\big(\F^f(R)\big)\ar[r] & Y\big(\F^f(R)\big)
}$$
where the horizontal arrows are given by pre-composition with $f$ itself, while the vertical arrows are given by pre-composition with $\spec\big(\F^f(R)\big)\to\spec\big(\E^e(R)\big)$.
It is now obvious that this square commutes.

For part~(c), 
to say that $\alpha$ is a closed embedding means that there is a sub-$\mathbb{S}$-algebra scheme ${\mathcal B} \leq {\mathcal F}$  for which $\alpha$ induces an isomorphism between ${\mathcal E}$ and ${\mathcal B}$.
As ${\mathcal F}$ is affine over $A$, ${\mathcal B}$ is defined as the kernel of some map of group schemes $\beta:{\mathcal F} \to {\mathbb G}_a^N$ for some $N$.
Now, for any $k$-algebra $R$, because $\alpha$ is an isomorphism between ${\mathcal E}^e(R)$ and ${\mathcal B}^f(R)$, $\alpha$ induces an identification of $X\big({\mathcal E}^e(R)\big)$ with
$X\big({\mathcal B}^f(R)\big)$.
That is, $\widehat{\alpha}$ is an isomorphism between $\tau(X,\E,e)$ and $\tau(X,\B,f)$.
Note that $X({\mathcal B}^f(R))$ consists of those ${\mathcal F}^f(R)$-valued points of $X$ which happen to belong to ${\mathcal B}^f(R)$ and this set is cut out by $\beta$.
These give us the equations expressing $\tau(X,\B,f)$ as a closed subscheme of $\tau(X,\F,f)$.
\end{proof}

\begin{lemma}
\label{rcompared}
The following diagram commutes.
$$\xymatrix{
\tau(X,\E,e)\times_k\F(k)\ar[rr]^{\hat\alpha\otimes\id}\ar[d]^{\id\otimes\underline{\alpha}} && \tau(X,\F,f)\times_k\F(k)\ar[d]^{r^X_{\F,f}}\\
\tau(X,\E,e)\times_k\E(k)\ar[rr]^{r^X_{\E,e}} && X
}$$
where $r^{\cdot}_{\cdot,\cdot}$ are the canonical morphisms of Definition~\ref{canonicalmorph}, associated to the respective prolongations.
\end{lemma}

\begin{proof}
Identifying $\E=\mathbb{S}^{\ell_\E}$ and $\F=\mathbb{S}^{\ell_\F}$ we take $e_0,\dots,e_{\ell_\E-1}$ to be the standard basis for $\E$ and $f_0,\dots,f_{\ell_\F-1}$ to be the standard basis for $\F$.
We write $\alpha=(\alpha_1,\dots,\alpha_{\ell_\F})$, where the $\alpha_1,\dots,\alpha_{\ell_\F}$ are linear polynomials in $\ell_\E$-variables.

Covering $X$ by affine open subsets, and using functoriality of the maps involved,
it is not hard to see that it suffices to consider affine space $X=\spec\big(k[\bf x]\big)$, where $\bf x$ is a (possibly infinite) tuple of indeterminates.
We have identifications
$$\tau(X,\E,e)=\spec\big(k[\overline{ \bf y}]\big)$$
where $\overline{\bf y} =({\bf y}_j)_{0\leq j\leq \ell_\E-1}$ is $\ell_\E$ copies of $\bf x$,
and
$$\tau(X,\F,f)=\spec\big(k[\overline{ \bf z}]\big)$$
where $\overline{ \bf z} =({\bf z}_{j'})_{0\leq j'\leq \ell_\F-1}$ is $\ell_\F$ copies of $\bf x$.
Taking global sections, we need to show that the following diagram commutes:
$$\xymatrix{
k[\bf x ]\ar[rr]^{r^*_{\F,f}}\ar[d]^{r^*_{\E,e}} && k[\overline{\bf z}]\otimes_k\F(k)\ar[d]^{\alpha^*\otimes\id}\\
k[\overline{\bf y}]\otimes_k\E(k)\ar[rr]^{\id\otimes\alpha} && k[\overline{\bf y}]\otimes_k\F(k)
}$$
It is clear that these maps commute on constants, so it sufffices to fix $x\in{\bf x}$ and chase it.
\begin{eqnarray*}
\alpha^*\otimes\id\big(r^*_{\F,f}(x)\big)
&=&
\alpha^*\otimes\id\big(\sum_{j'=0}^{\ell_\F-1}z_{j'}\otimes f_{j'}\big)\\
&=&
\sum_{j'=0}^{\ell_\F-1}\alpha_{j'}(y_{0},\dots,y_{\ell_\E-1})\otimes f_{j'}
\end{eqnarray*}
On the other hand,
\begin{eqnarray*}
\id\otimes\alpha\big(r^*_{\E,e}(x)\big)
&=&
\id\otimes\alpha\big(\sum_{j=0}^{\ell_\E-1}y_{j}\otimes e_{j}\big)\\
&=&
\sum_{j=0}^{\ell_\E-1}y_{j}\otimes\alpha(e_j)\\
&=&
\sum_{j=0}^{\ell_\E-1}y_{j}\otimes \sum_{j'=0}^{\ell_\F-1}\alpha_{j'}(e_j)f_{j'}\\
&=&
\sum_{j'=0}^{\ell_\F-1}\big(\sum_{j=0}^{\ell_\E-1}\alpha_{j'}(e_j)y_{j}\big)\otimes f_{j'}\\
&=&
\sum_{j'=0}^{\ell_\F-1}\alpha_{j'}(y_{0},\dots,y_{\ell_\E-1})\otimes f_{j'}
\end{eqnarray*}
as desired.
\end{proof}

\bigskip
\subsection{Composing prolongations}
\label{compose}
Fix two finite free $\mathbb{S}$-algebra schemes $\mathcal{E}$ and $\F$ over a ring $A$.
For any $A$-algebra $R$, the $R$-algebra structure on $\F(R)$ makes it into an $A$-algebra as well,
and hence it makes sense to consider $\E\big(\F(R)\big)$.
This inherits an $R$-algebra structure via
$$\xymatrix{
R\ar[rr]^{s_\F^R} && \F(R)\ar[rr]^{s_\E^{\F(R)}\quad} && \E\big(\F(R)\big)
}$$
Moreover, $\E\big(\F(R)\big)$ is thereby finite and free over $R$ witnessed by the $R$-linear isomorphism
$$\xymatrix{
 \E\big(\F(R)\big)\ar[rr]^{\psi_\E^{\F(R)}} &&\F(R)^{\ell_\E}\ar[rr]^{(\psi_\F^R)^{\ell_\E}}&& (R^{\ell_\F})^{\ell_\E}
 }$$
Let $\E\F$ denote the corresponding finite free $\mathbb{S}$-algebra scheme.
So for any $A$-algebra $R$, $\E\F(R)=\E\big(\F(R)\big)$, and $s_{\E\F}^R$ and $\psi_{\E\F}^R$ are the above displayed compositions.

\begin{remark}
\label{dmn}
Note that $\E\F$ is canonically isomorphic to $\E\otimes_\mathbb{S}\F$, and hence to $\F\E$, as an $\mathbb{S}$-algebra scheme.
Indeed, this is just Remark~\ref{presentation}: given an $A$-algebra $R$, $\E\big(\F(R)\big)$ is canonically identified with $\E(R)\otimes_R\F(R)$.
The induced isomorphism between $\F\big(\E(R)\big)$ and $\E\big(\F(R)\big)$ can be described in co-ordinates by
$$\xymatrix{
\F\big(\E(R)\big)\ar[dd]_{\psi_{\F\circ\E}^R} && \E\big(\F(R)\big)\\
\\
(R^{\ell_\E})^{\ell_\F}\ar[rr] && (R^{\ell_\F})^{\ell_\E}\ar[uu]_{(\psi_{\E\F}^R)^{-1}}
}$$
where $(R^{\ell_\E})^{\ell_\F}\to (R^{\ell_\F})^{\ell_\E}$ is the natural co-ordinate change.
\end{remark}

Now fix an $A$-algebra $k$ euipped with an $\mathcal{E}$-ring structure
$(k,e)$
and an $\mathcal{F}$-ring structure $(k,f)$.
Consider the $\E\F$-ring structure on $k$ given by the composition of $e$ with $\E(f)$,
$$\xymatrix{k \ar[rr]^{e} && \E(k) \ar[rr]^{\E(f)} && \E\big(\F(k)\big)}.$$We denote this homomorphism by $ef:k\to\E\F(k)$.

\begin{lemma}
\label{rem-emcircen}
The $k$-algebras $\E\F^{ef}(k)$ and $\E^e\big(\F^f(k)\big)$ are naturally isomorphic.
More precisely, there is a canonical ring isomorphism $\gamma:\E\F(k)\to\E(k)\otimes_k\F^f(k)$ such that the following commutes:
$$\xymatrix{
\E\F(k)\ar[rr]^{\gamma \ \ \ } && \E(k)\otimes_k\F^f(k)\\
k\ar[u]^{ef}\ar[rr]^{e} && \E(k)\ar[u]
}$$
\end{lemma}

\begin{proof}
By Remark~\ref{presentation}, with $R=\F^f(k)$, we have a canonical identification of $\E\big(\F^f(k)\big)$ with $\E(k)\otimes_k\F^f(k)$.
On the other hand $\F(k)$ and $\F^f(k)$ are identical as $A$-algebras and hence $\E\big(\F(k)\big)$ and $\E\big(\F^f(k)\big)$ are canonically isomorphic as rings.
Hence we obtain a canonical ring isomorphism $\gamma:\E\big(\F(k)\big)\to\E(k)\otimes_k\F^f(k)$.
Now, since $ef=\E(f)\circ e$, the desired commuting square reduces to showing that
$$\xymatrix{
\E\big(\F(k)\big)\ar[rr]^{\gamma \ \ \ } && \E(k)\otimes_k\F^f(k)\\
& \E(k)\ar[ul]^{\E(f)}\ar[ur]
}$$
commutes.
But this is just the fact that $\E\big(\F^f(k)\big)$ identifies with $\E(k)\otimes_k\F^f(k)$ even as an $\E(k)$-algebra (cf. Remark~\ref{presentation}).
\end{proof}

We wish to describe, in terms of Weil restriction, the composition of the $\E$- and $\F$-prolongations.

\begin{proposition}
\label{taumtaun}
Suppose $X$ is a scheme over $k$.
Then
$$\tau\big(\tau(X,\E,e),\F,f\big)=\tau(X,\E\F,ef)$$
That is, the $\F$-prolongation of the $\E$-prolongation of $X$ is the Weil restriction of $X\times_{ef}\E\F(k)$ from $\E\F(k)$ to $k$.
Moreover, under this identification,
$$\nabla_{\F,f}\circ\nabla_{\E,e}:X(k)\to\tau\big(\tau(X,\E,e),\F,f\big)(k)$$
becomes
$$\nabla_{\E\F,ef}:X(k)\to\tau(X,\E\F, ef)(k)$$
\end{proposition}

\begin{proof}
For ease of notation, let $S:=\E(k)\otimes_k\F^f(k)$.
Given any scheme $U$ over $k$
\begin{eqnarray*}
\tau^\F\big(\tau^\E(X)\big)(U)
& = &
\r_{\F(k)/k}\big(\r_{\E(k)/k}(X\times_k\E^e(k))\times_{f}\F(k)\big)(U)\\
& = &
\r_{\F(k)/k}\big(\r_{S/\F(k)}\big(X\times_k\E^e(k)\times_{\E(k)} S\big)\big)(U)\\
& = &
\hom_{\F(k)}\big(U\times_k\F(k),\r_{S/\F(k)}\big(X\times_k\E^e(k)\times_{\E(k)} S\big)\big)\\
& = &
\hom_{S}\big(U\times_k S, X\times_k\E^e(k)\times_{\E(k)} S\big)
\end{eqnarray*}
where the second equality is by the compatibility of Weil restrictions with base change (cf. Fact~\ref{weilbc}).
Now, applying the ring isomorphism $\gamma:\E\big(\F(k)\big)\to S$ given by Lemma~\ref{rem-emcircen},
and keeping in mind the commuting square given by that lemma,
we see that this last representation of $\tau^\F\big(\tau^\E(X)\big)(U)$ identifies with
$$\hom_{\E\F(k)}\big(U\times_k\E\F(k),X\times_{ef}\E\F(k)\big)$$
which is $\r_{\E\F(k))/k}\big(X\times_{ef}\E\F(k)\big)(U)$
as desired.

For the moreover clause, first note that under the identification
$$\tau\big(\tau(X,\E,e),\F,f\big)(k) = \tau(X,\E,e)\big(\F^f(k)\big)$$
$\nabla_{\F,f}$ is by definition precomposition with $f^*:\spec(k)\to\spec\big(\F^f(k)\big)$.
Then, under the further identification of $\tau(X,\E,e)(k)$ with $X\big(\E^e(k)\big)$ and of 
$\tau(X,\E,e)\big(\F^f(k)\big)$ with
$X\big(\E^e(\F^f(k))\big)$, it is not hard to see that $\nabla_{\F,f}$ becomes precomposoition with $\big(f\otimes_k\E(k)\big)^*$.
Hence $\nabla_{\F,f}\circ\nabla_{\E,e}$, viewed as a map from $X(k)$ to $X\big(\E^e(\F^f(k))\big)$ is just precomposition with $\big((f\otimes_k\E(k))\circ e\big)^*$.
Now, applying $\gamma$ from Lemma~\ref{rem-emcircen}, $\big(f\otimes_k\E(k)\big)\circ e:k\to \E(k)\otimes_k\F^f(k)$ transforms to $ef:k\to\E\F(k)$, and precomposition with $(ef)^*$ is by definition $\nabla_{\E\F,ef}$.
So $\nabla_{\F,f}\circ\nabla_{\E,e}$ identifies with $\nabla_{\E\F,ef}$.
\end{proof}

\begin{corollary}
\label{commutingprolong} Prolongation spaces commute.
That is,
$\tau\big(\tau(X,\E,e),\F,f\big)$ is canonically isomorphic to $\tau\big(\tau(X,\F,f),\E,e\big)$
\end{corollary}

\begin{proof}
This is immediate from Proposition~\ref{taumtaun} and Remark~\ref{dmn}.
\end{proof}

\begin{lemma}
\label{rcomposed}
The following diagram commutes.
$$\xymatrix{
\tau(X,\E\F,ef)\times_k\E\F(k)\ar[dd]_{r^{\tau(X,\E,e)}_{\F,f}\times_k\E(k)}\ar[drr]^{r^X_{\E\F,ef}}\\
&&X\\
\tau(X,\E,e)\times_k\E(k)\ar[urr]_{r^X_{\E,e}}
}$$
where $r^{\cdot}_{\cdot,\cdot}$ are the canonical morphisms of Definition~\ref{canonicalmorph} associated to the respective prolongations. 
\end{lemma}

\begin{proof}
We make the usual abbreviation of $\tau^\E(X)$ for $\tau(X,\E,e)$.

\begin{remark}
\label{whatvertmeans}
Since $r^{\tau^\E(X)}_{\F,f}:\tau^\F\big(\tau^\E(X)\big)\times_k\F(k)\to\tau^\E(X)$ is not a morphism over $k$ in the usual sense (cf. Remark~\ref{rnotoverk}), we need to make clear what we mean by $r^{\tau^\E(X)}_{\F,f}\times_k\E(k)$.
First of all, we use Proposition~\ref{taumtaun} to identify $\tau^{\E\F}(X)\times_k\E\F(k)$ with $\tau^\F\big(\tau^\E(X)\big)\times_k\E\big(\F(k)\big)$.
Now consider the $\F(k)$-morphism $\tau^\F\big(\tau^\E(X)\big)\times_k\F(k)\to\tau^\E(X)\times_k\F^f(k)$ given by the Weil restriction.
We can base change it up to $\E\big(\F(k)\big)$ to get an $\F(k)$-morphims
$$a:\tau^\F\big(\tau^\E(X)\big)\times_k\F(k)\times_{\F(k)}\E\big(\F(k)\big)
\longrightarrow
\tau^\E(X)\times_k\F^f(k)\times_{\F(k)}\E\big(\F(k)\big)$$
Next, considering the following commuting diagram of $k$-algebras
$$\xymatrix{
\F(k)\ar[r]&\E\big(\F(k)\big)\\
k\ar[u]^f\ar[r] &\E(k)\ar[u]^{\E(f)}
}$$
we have that
$\tau^\E(X)\times_k\F^f(k)\times_{\F(k)}\E\big(\F(k)\big)
= \tau^\E(X)\times_k\E(k)\times_{\E(f)}\E\big(\F(k)\big)$, and so we can compose $a$ with the natural projection
$$b:\tau^\E(X)\times_k\E(k)\times_{\E(f)}\E\big(\F(k)\big)\to \tau^\E(X)\times_k\E(k)$$
So $r^{\tau^\E(X)}_{\F,f}\times_k\E(k):\tau^\F\big(\tau^\E(X)\big)\times_k\E\big(\F(k)\big)\to\tau^\E(X)\times_k\E(k)$ is $b\circ a$.
\end{remark}

Now we proceed with the proof of Lemma~\ref{rcomposed}.
Identifying $\E=\mathbb{S}^{\ell_\E}$ and $\F=\mathbb{S}^{\ell_\F}$ we take $e_0,\dots,e_{\ell_\E-1}$ to be the standard basis for $\E$, $f_0,\dots,f_{\ell_\F-1}$ to be the standard basis for $\F$, and $\{e_j\otimes f_{j'}:0\leq i\leq\ell_E,0\leq j'\leq\ell_\F\}$ the corresponding standard basis for $\E\F=\E\otimes_\mathbb{S}\F$.

Covering $X$ by affine open subsets, and using functoriality of the Weil restrictions,
it is not hard to see that it suffices to consider $X=\mathbb{A}_k^r=\spec\big(k[\bf x ] \big)$, where $\bf x$ is a (possibly infinite) tuple of indeterminates.
We have identifications
$$\tau(X,\E,e)=\spec\big(k[\overline{\bf y}]\big)$$
where $\overline{\bf y}=(y_{j})_{0\leq j\leq \ell_\E-1}$ is $\ell_\E$ copies of $\bf x$,
and
$$\tau(X,\E\F,ef)=\spec\big(k[\overline{\bf z}]\big)$$
where $\overline{\bf z}=(z_{j,j'})_{0\leq j\leq \ell_\E-1,0\leq j'\leq \ell_\F-1}$ is $\ell_\E\ell_\F$ copies of $\bf x$.
Taking global sections we need to show that the following diagram commutes:
$$\xymatrix{
&&k[\overline{\bf z}]\otimes_k\E\F(k)\\
k[\bf x]\ar[urr]_{(r^X_{\E\F,ef})^*}\ar[drr]_{(r^X_{\E,e})^*\ \ }&&\\
&&k[\overline{\bf y}]\otimes_k\E(k)\ar[uu]_{\big(r^{\tau^\E(X)}_{\F,f}\times_k\E(k)\big)^*}
}$$
Let us first check this on constants $a\in k$.
Going clockwise we have $a\mapsto 1\otimes ef(a)$.
Going counter-clockwise we have that $a\mapsto 1\otimes e(a)\mapsto \big(r^{\tau^\E(X)}_{\F,f}\times_k\E(k)\big)^*\big(1\otimes e(a)\big)$
Now, using the explanation of what $r^{\tau^\E(X)}_{\F,f}\times_k\E(k)$ is in Remark~\ref{whatvertmeans}, and the fact that $\E(f)\big(e(a)\big)=ef(a)$ by definition, we see that $\big(r^{\tau^\E(X)}_{\F,f}\times_k\E(k)\big)^*\big(1\otimes e(a)\big)=1\otimes ef(a)$.
So the diagram commutes on constants.

It remains to chase fixed $x\in\bf x$.
Going clockwise we have $\displaystyle x\mapsto \sum_{j, j'}z_{j,j'}\otimes\big(e_j\otimes f_{j'}\big)$.
Going counter-clockwise, $\displaystyle x\mapsto \big(r^{\tau^\E(X)}_{\F,f}\times_k\E(k)\big)^*\big(\sum_jy_{j}\otimes e_j\big)$.
Since $\displaystyle (r^{\tau^\E(X)}_{\F,f})^*(y_{j})=\sum_{j'}z_{j,j'}\otimes f_{j'}$, we have that the diagram commutes on each $x\in\bf x$, and hence commutes.
\end{proof}

\bigskip
\subsection{The structure of the prolongation space}
In this section we specialise to the case when $k$ is a field and in this case describe the structure of the prolongation space.
Fix a finite free $\mathbb{S}$-algebra schemes $\E$ and an $\E$-field $(k,e)$.

\begin{proposition}[$k$ a field]
\label{reducetolocal}
There exist finite free {\em local} $\mathbb{S}$-algebra schemes $\F_1,\dots, \F_t$ with ring homomorphisms $f_i:k\to\F_i(k)$ such that
$\displaystyle \tau(X,\E,e)=\prod_{i=1}^t\tau(X,\F_i,f_i)$, for any scheme $X$ over $k$.
\end{proposition}

\begin{proof}
The ring $\E(k)$ is an artinian $k$-algebra and hence can be expressed as a finite product of local artinian $k$-algebras, say $B_1,\dots, B_t$.
After choosing bases, we obtain, for each $i=1,\dots, t$, finite free $\mathbb{S}$-algebra schemes $\F_i$ over $A$ such that $\F_i(R)=B_i\otimes_kR$ for any $k$-algebra $R$.
Let $f_i:k\to\F_i(k)$ be the composition of $e:k\to\E(k)$ with the projection $\E(k)\to B_i=\F_i(k)$.
It is not hard to see, using Remark~\ref{presentation}, that for any $k$-algebra $R$, $\displaystyle \E(R)=\prod_{i=1}^t\F_i(R)$ and $\displaystyle \E^e(R)=\prod_{i=1}^t\F_i^{f_i}(R)$.
Hence, $\displaystyle \tau(X,\E,e)(R)=X\big(\E^e(R)\big)=X\big(\prod_{i=1}^t\F_i^{f_i}(R)\big)=\prod_{i=1}^tX\big(\F_i^{f_i}(R)\big)=\big(\prod_{i=1}^t\tau(X,\F_i,f_i)\big)(R)$.
These identifications being functorial in all $k$-algebras $R$, we get $\displaystyle \tau(X,\E,e)=\prod_{i=1}^t\tau(X,\F_i,f_i)$, as desired.
\end{proof}

Proposition~\ref{reducetolocal} largely reduces the study of prolongation spaces (over fields) to the case when $\E$ is {\em local}.
The following proposition describes the structure of the prolongation space in the local case under the additional hypothesis that the residue field is the base field.

\begin{proposition}
\label{local-case}
Suppose $k$ is field and $\E(k)$ is a local (finite free) $k$-algebra with maximal ideal $\mathfrak{m}$ and such that $\E(k)/\mathfrak{m}=k$.
Let $d$ be greatest such that $\mathfrak{m}^d\neq(0)$ (by artinianity).
Consider the sequence
$$\E=\E_d\to\E_{d-1}\to\cdots\to\E_1\to\E_0=\mathbb{S}$$
where $\E_i(k)=\E(k)/\mathfrak{m}^{i+1}$ and $\rho_i:\E_{i+1}(k)\to\E_i(k)$ is the quotient map.
Let $e_i:=\rho_i\circ\cdots\circ\rho_{d-1}\circ e:k\to\E_i(k)$ be the induced $\E_i$-field structures on $k$.
(So $e_d=e$ and $e_0$ is an endomorphism of $k$.)

Let $X$ be a smooth and absolutely irreducible scheme over $k$, and consider the induced sequence of morphisms
$$\tau(X,\E,e)\to\tau(X,\E_{d-1},e_{d-1})\to\cdots\to\tau(X,\E_1,e_1)\to X^{e_0}$$
where $X^{e_0}=\tau(X,\E_0,e_0)$ is the transform of $X$ by $e_0$.
Then for each $i=0,\dots,d-1$,
\begin{itemize}
\item[$(a)_i$] 
$\tau(X,\E_i,e_i)$ is smooth and absolutely irreducible, and
\item[$(b)_i$]
$\tau(X,\E_{i+1},e_{i+1})\to \tau(X,\E_i,e_i)$ is a torsor for a power of the tangent bundle of $X^{e_0}$.
That is, letting $m=\dim_K\big(\mathfrak{m}^{i+2}/\mathfrak{m}^{i+1}\big)$, there is a morphism
$$
\xymatrix{
(TX^{e_0})^m\times_{X^{e_0}}\tau(X,\E_{i+1},e_{i+1})\ar[dr] \ar[rr]^{\ \ \ \ \ \ \ \ \gamma} && \tau(X,\E_{i+1},e_{i+1})\ar[dl]\\
& \tau(X,\E_i,e_i)
}$$
such that for every $b\in \tau(X,\E_i,e_i)$ with image $a\in X^{e_0}$, $\gamma_b$ defines a principal homogeneous action of $(T_aX^{e_0})^m$ on $\tau(X,\E_{i+1},e_{i+1})_b$.
\end{itemize}
In particular, $\tau(X,\E,e)$ is smooth and absolutely irreducible.
\end{proposition}

\begin{proof}
Note that since $\E_0(k)=k$, $e_0$ is an endomorphism of $k$, and so $\tau(X,\E_0,e_0)=X^{e_0}$ by Example~\ref{prolongexamples}(c).
Hence $\tau(X,\E_0,e_0)$ is smooth and absolutely irreducible.
Now observe that $(a)_0$ together with $(a)_i$ and $(b)_i$ imply $(a)_{i+1}$.
Hence it will suffice to show that $(a)_i$ implies $(b)_i$.

We assume that $\tau(X,\E_i,e_i)$ is smooth and define the action $\gamma$ uniformly on the fibres.
We will work at the level of $R$-points where $R$ is a fixed arbitrary $k$-algebra.
Let $I:=\ker\big(\E(R)\to \E_0(R)=R\big)$.
Then $I^{i+2}=\ker\big(\E(R)\to\E_{i+1}(R)\big)$.
Note that $(\mathfrak{m}^{i+1}/\mathfrak{m}^{i+2})\otimes_kR=I^{i+2}/I^{i+1}$.
So if $v_1,\dots, v_m$ is a $k$-basis of $(\mathfrak{m}^{i+1}/\mathfrak{m}^{i+2})$ then $I^{i+2}/I^{i+1}$ is a free $R$-module with basis $w_1=v_1\otimes_k1_{R},\dots,w_m=v_m\otimes_k1_{R}$.

Working locally we may assume that $X$ is affine, defined by a sequence of polynomials $g:=(g_1\dots,g_r)$ over $k$.
Let $b\in \tau(X,\E_i,e_i)(R)=X\big(\E_i^{e_i}(R)\big)$ with image $a\in \tau(X,\E_0,e_0)(R)=X^{e_0}(R)=X\big(\E_0^{e_0}(R)\big)$.
By smoothness of $\tau(X,\E_i,e_i)$ we can lift $b$ to an element of $X\big(\E_{i+1}^{e_{i+1}}(R)\big)$.
Let $c$ be any such lifting of $b$.
We view $c$ as a tuple of elements in $\E(R)$ representing elements of $\E(R)/I^{i+2}$ such that $c=b\mod I^{i+1}$ and $g^e(c)=0\mod I^{i+2}$, where $g^e$ denotes the sequence of polynomials over $\E(k)$ obtained from $g$ by applying $e$ to the coefficients.
Given $y=(y_1,\dots,y_m)\in (T_aX^{e_0})^m(R)$ we need to define $\gamma(y,c)$.
Observe that $\operatorname{d}g_a^{e_0}(y_j)=0$ for all $j=1,\dots,m$, where $g^{e_0}$ denotes the sequence of polynomials obtained from $g$ by applying $e_0$ to the coefficients.
Hence if we set 
$\gamma(y,c):=(c+y_1w_1+\cdots y_mw_m)\mod I^{i+2}$
then $\gamma(y,c)$ also lifts $b$ and 
\begin{eqnarray*}
g^e\big(\gamma(y,c)\big)
& = &
g^e(c+y_1w_1+\cdots y_mw_m)\mod I^{i+2}\\
&=&
g^e(c)+\big(\sum_{j=1}^m\operatorname{d}g_a^{e_0}(y_j)w_j\big)\mod I^{i+2}\\
&=&0\mod I^{i+2}
\end{eqnarray*}
so that $\gamma(y,c)\in X\big(\E_{i+1}^{e_{i+1}}(R)\big)$.
Conversely,
if $c$ and $c'$ both lift $b$ then $c'=c\mod I^{i+1}$ and so $c'=(c+y_1w_1+\cdots y_mw_m)\mod I^{i+2}$ for some $y_1,\dots,y_m\in R$.
The same calculation as above shows that each $y_j$ is an $R$-point of the tangent space to $X^{e_0}$ at $a$.
That is, $\gamma$ defines a principal homogeneous action of $(T_aX^{e_0})^m(R)$ on $\tau(X,\E_{i+1},e_{i+1})_b(R)$.
\end{proof}

\begin{corollary}[$k$ a field]
\label{prolongsmooth}
The prolongation space of a smooth and absolutely irreducible scheme is itself smooth and absolutely irreducible.
\end{corollary}

\begin{proof}
First of all, by Proposition~\ref{reducetolocal} it suffices to consider the case when $\E(k)$ is a local $k$-algebra.
If the residue field is $k$ then Proposition~\ref{local-case} describes the complete structure of the prolongation space, and shows in particular that it is smooth and absolutely irreducible.
In general, the residue field, $\E_0(k)$ in the notation of~\ref{local-case}, may be a finite extension of $k$.
The proof of Proposition~\ref{local-case} still goes through except for the description of the base $\tau(X,\E_0,e_0)$.
That is, we obtain the same description of the fibrations $\tau(X,\E_{i+1},e_{i+1})\to\tau(X,\E_i,e_i)$ for $i=0,\dots, d-1$, but now as torsors for a power of the tangent bundle of $\tau(X,\E_0,e_0)$.
Hence, to prove the Corollary we need only prove that $\tau(X,\E_0,e_0)$ is smooth and absolutely irreducible.
But by definition $\tau(X,\E_0,e_0)=\operatorname{R}_{\E_0(k)/k}\big(X\times_k\E_0^{e_0}(k)\big)$.
Now base change to a field extension preserves smoothness and absolute irreducibility, and in general Weil restrictions preserve smoothness (cf. Fact~\ref{weiletale}).
The Corollary then follows from the fact that if $K$ is a finite field extension of $k$ and $Y$ is a smooth and absolutely irreducible scheme over $K$ then the Weil restriction $\operatorname{R}_{K/k}(Y)$ is absolutely irreducible.
\end{proof}

\begin{remark}
As can be seen already at the level of tangent spaces, without smoothness, absolute irreducibility is not necessarily preserved under prolongations.
\end{remark}

\bigskip

\section{Algebraic jet spaces}

\noindent
There are various kinds of ``jet spaces" for algebraic varieties in the literature.
We will settle on one notion, essentially the linear space associated with the sheaf of differentials,
and recall a number of its fundamental properties.
With the reader unfamiliar with this literature in mind, we will provide proofs of these well-known results.

\medskip

We begin with some simple observations in commutative algebra.

\begin{lemma}
\label{kernelpower}
Given a ring $C$ suppose $R$ and $B$ are $C$-algebras and $D$ is an ideal of $R$.
For any $n\in\mathbb{N}$, let $\pi_n:R\otimes _CB\to \big(R/D^{n}\big)\otimes_CB$ be the quotient map tensored with $B$.
Then $\ker(\pi_n)=(\ker\pi_1)^n$.
\end{lemma}

\begin{proof}
This is a straightforward computation using the fact that $\ker(\pi_n)$ is generated by elements of the form $r\otimes b$ where $r\in D^n$.
\end{proof}

\begin{lemma}
\label{ringidentity}
Fix a ring $A$ and $A$-algebras $B$ and $C$, together with a map of $A$-algebra $\alpha:C \to B$.
Let $J$ be the kernel of $\alpha\otimes\id:C \otimes_A B \to B$.
We regard $C \otimes_A C$ as a $C$-algebra via multiplication on the left.
Let $D$ be the kernel of the $C$-algebra map $C \otimes_A C \to C$ given by $x \otimes y \mapsto x \cdot y$.
Then the $B$-algebra isomorphism $f:(C \otimes_A C) \otimes_C B \to C \otimes_A B$ given by $(x \otimes y) \otimes b \mapsto y \otimes \alpha(x) \cdot b$ induces a $B$-linear isomorphism between $\big(D/D^{n+1}\big)\otimes_C B$ and $J/J^{n+1}$.
 \end{lemma}

\begin{proof}
A simple diagram chase shows that the following diagram commutes:
$$\xymatrix{
(C \otimes_A C) \otimes_C B\ar[r]^f\ar[d]_\pi & C\otimes_AB\ar[d]^\rho\\
[(C \otimes_A C)/D] \otimes_C B\ar[r]^{\ \ \approx} & (C \otimes_A B)/J
}$$
where $\pi$ is the quotient map tensored with $B$, $\rho$ is the quotient map, and the bottom isomorphism is the natural identification of both rings with $B$.
It follows that if we let $\tilde{D}=\ker\pi$ then $f(\tilde{D})=J$.
So $f(\tilde{D}^{n+1})=J^{n+1}$ for any $n$.
On the other hand,
$$\tilde{D}^{n+1}=\ker\big(C \otimes_A C) \otimes_C B\to [(C \otimes_A C)/D^{n+1}] \otimes_C B\big)$$
by Lemma~\ref{kernelpower}.
It follows that $f$ induces a $B$-algebra isomorphism between $\big[(C\otimes_AC)/D^{n+1}\big]\otimes_CB$ and $(C\otimes_AB)/J^{n+1}$, which restricts to an isomorphism between $\big(D/D^{n+1}\big) \otimes_C B$ and $J/J^{n+1}$, as desired.
\end{proof}

\bigskip
Let us now fix a scheme $X$ over a ring $k$.

\begin{definition}
\label{jetspace}
Consider $\O_X \otimes_k \O_X$ as an $\O_X$-algebra via $f \cdot (g \otimes h) := (fg) \otimes h$. 
\begin{itemize}
\item[(a)]
By the sheaf of {\em $n^\text{th}$ co-jets on $X$} we will mean the coherent $\O_X$-module
$\scoj_X^{(n)} := \I/\I^{n+1}$ where $\I$ is the kernel of the $\O_X$-algebra map $\O_X \otimes_k \O_X \to \O_X$ given on sections by $f \otimes g \mapsto f \cdot g$.
\item[(b)]
The {\em $n^\text{th}$ jet space over $X$}, denoted by $\jet^n(X)$, is the linear space associated to $\scoj_X^{(n)}$.
\end{itemize}
That is, $\jet^n(X)\to X$ respresents the functor which associates to every $X$-scheme $g:Y\to X$ the set $\hom_{\O_Y}(g^*\scoj_X^{(n)},\O_Y)$.\footnote{Some discussion of the terminology here is warranted.
Our jet space is closely related to, but different from, Kantor's~\cite{kantor} ``sheaf of jets'': he does not take the associated linear space, but rather works with the sheaf itself. Moreover, his sheaf of jets is $(\O_X\otimes_k\O_X)/\I^{n+1}$ while our sheaf of co-jets is $\I/\I^{n+1}$.
Our jet spaces also differ, more seriously, with Buium's~\cite{buium92} ``jet spaces'', which are what we have called arc spaces (in the pure algebraic setting) and what we have called prolongation spaces (in the differential setting).
Our co-jets coincide in the smooth case with the sheaf of differentials of Grothendieck~\cite{groth}.}
\end{definition}

We recall the (local) construction of the the linear space associated to a coherent sheaf $\mathcal{F}$ on a scheme $X$.
Working locally, let us assume that $X=\spec(A)$ and that $\mathcal{F}$ has a finite presentation over $X$ given by the exact sequence
$$\xymatrix{
A^p \ar^a[r] & A^q \ar^b[r] & \Gamma(X,\mathcal{F})\ar[r] & 0
}$$
Writing $a=(a_{ij})$ as a $p\times q$ matrix over $A$, $x_j\mapsto\sum_{i=1}^qa_{ij}y_i$ determines a map $A[x_1,\dots,x_p]\to A[y_1,\dots,y_q]$.
Taking spectra we obtain a map of group schemes
$$\xymatrix{
 (\mathbb{G}_X^a)^q \ar^{a^*}[rr] \ar[rd] & &(\mathbb{G}_X^a)^p \ar[ld] \\
 & X
  }$$
The linear space $L(\mathcal{F})\to X$ associated to $\mathcal{F}$ is the kernel of $a^*$ as a group subscheme of $(\mathbb{G}_X^a)^q\to X$.
Note that in the case when $\mathcal{F}$ is locally free, the linear space associated to $\mathcal{F}$ is {\em dual} to the vector bundle associated to $\mathcal{F}$.

\begin{remark}
\begin{itemize}
\item[(a)]
If $X$ is a smooth and irreducible then so is $\jet^n(X)$.
\item[(b)]
Jet spaces commute with base change:
$\jet^n(X\times_kR)=\jet^n(X)\times_kR$
for all $k$-algebras $R$.
\end{itemize}
\end{remark}

\begin{proof}
Smoothness of $X$ implies that $\scoj^{(n)}_X$ is a locally free sheaf on $X$, and so the linear space associated to it is dual to the associated vector bundle.
In particular, if $X$ is smooth then $\jet^{(n)}_X$ is a vector bundle over $X$.

For part~(b),
setting $f:X\times_kR\to X$ to be the projection, we first observe that $\scoj^{(n)}_{X\times_kR}=f^*\scoj^{(n)}_X$.
Fix $U$ an affine open set in $X$ and let $\rho$ be the natural quotient map
$$\big[\O_X(U)\otimes_k\O_X(U)\big]\otimes_kR\longrightarrow \big[\big(\O_X(U)\otimes_k\O_X(U)\big)/\I(U)\big]\otimes_kR$$
Identifying
$\big[\O_X(U)\otimes_k\O_X(U)\big]\otimes_kR=\big(\O_X(U)\otimes_kR\big)\otimes_R\big(\O_X(U)\otimes_kR\big)\big)$
and 
$\big[\big(\O_X(U)\otimes_k\O_X(U)\big)/\I(U)\big]\otimes_kR=\O_X(U)\otimes_kR$,
we see that $\rho$ coincides with $(a\otimes r)\otimes(b\otimes s)\longmapsto ab\otimes rs$.
Hence, using Lemma~\ref{kernelpower}, we see that
\begin{eqnarray*}
\scoj^{(n)}_{X\times_kR}(U\times_kR)
&=&
\ker\rho/(\ker\rho)^{n+1}\\
&=&
\big[\I(U)/\I^{n+1}(U)\big]\otimes_kR\\
&=&
f^*\scoj^{(n)}_X(U\times_kR)
\end{eqnarray*}
So $\scoj^{(n)}_{X\times_kR}=f^*\scoj^{(n)}_X$.
Taking linear spaces of both sides, we get
$$\jet^n(X\times_kR)=\jet^n(X)\times_X(X\times_kR)=\jet^n(X)\times_kR$$
as desired
\end{proof}

\begin{definition}
\label{jetatp}
Given a scheme $S$ over $k$ and a morphism $p:S\to X$, we denote by $\jet^n(X)_p$ the scheme $\jet^n(X)\times_X S$ and we call it the {\em $n^\text{th}$ jet space of $X$ at $p$}.
\end{definition}

Note that $\jet^n(X)_p$ is the linear space over $S$ associated to the $\O_S$-module $p^*\scoj_X^{(n)}=p^{-1}\scoj_X^{(n)}\otimes_{p^{-1}\O_X}\O_S$.
The following proposition gives an alternative and useful presentation of $p^*\scoj_X^{(n)}$.

\begin{proposition}
\label{jetsatapoint}
Suppose $S$ is a scheme over $k$ and $p:S\to X$ is an $S$-point of $X$.
Let $\mathcal{J}_p$ be the kernel of the $\O_S$-algebra map $p^\sharp\otimes\id:p^{-1}\O_X\otimes_k\O_S\to\O_S$.

Then for each $n$, $p^*\scoj_X^{(n)}$ is naturally isomorphic to $\mathcal{J}_p/\mathcal{J}_p^{n+1}$.
It follows that
$$\jet^n(X)_p(S)=\hom_{\O_S}\big(\mathcal{J}_p/\mathcal{J}_p^{n+1},\O_S\big).$$
In particular, if $k$ is a field and $p\in X(k)$ then
$$\jet^n(X)_p(k)=\hom_k\big(\mathfrak{m}_p/\mathfrak{m}_p^{m+1},k\big)$$
where $\mathfrak{m}_p$ is the maximal ideal at (the topological point associated to) $p$.
\end{proposition}

\begin{proof}
We describe an isomorphism $p^*\scoj_X^{(n)}\to \mathcal{J}_p/\mathcal{J}_p^{n+1}$ on sections. 
Fix an open set $U$ in $S$ and an open set $V$ in $X$ containing $p(U)$.
Let $\alpha_V:\O_X(V)\to\O_S(U)$ be the composition of the map from $\O_X(V)$ to $\O_S\big(p^{-1}(V)\big)$ induced by $p$ together with the restriction from $\O_S\big(p^{-1}(V)\big)$ to $\O_S(U)$.
Note that $p^\sharp$ on $U$ is obtained as the direct limit of $\alpha_V$ as $V$ ranges over open subsets of $X$ containing $p(U)$.
Moreover, $\displaystyle \mathcal{J}_p(U)$ is the corresponding direct limit of $J_V:=\ker(\alpha_V\otimes\id)$.
Now Lemma~\ref{ringidentity} applied to $\big(A=k, B=\O_S(U), C=\O_X(V), \alpha=\alpha_V\big)$ yields a natural isomorphism
$\big[\I(V)/\I^{n+1}(V)\big]\otimes_{\O_X(V)}\O_S(U)\longrightarrow J_V/J_V^{n+1}$
Taking direct limits we obtain an isomorphism
$p^*\scoj_X^{(n)}(U)=p^{-1}\scoj_X^{(n)}(U)\otimes_{p^{-1}\O_X(U)}\O_S(U)\rightarrow
\mathcal{J}_p/\mathcal{J}_p^{n+1}(U)$.

As mentioned earlier, $\jet^n(X)_p$ is the linear space over $S$ associated to $p^*\scoj_X^{(n)}$.
Hence,
\begin{eqnarray*}
\jet^n(X)_p(S)
&=&
\hom_{\O_S}\big(p^*\scoj_X^{(n)},\O_S\big)\\
&=&
\hom_{\O_S}\big(\mathcal{J}_p/\mathcal{J}_p^{n+1},\O_S\big)
\end{eqnarray*}
as desired.

Finally, if $k$ is a field and $p\in X(k)$ then setting $S=\spec(k)$ and applying the above result yields that $\jet^n(X)_p(k)=\hom_k\big(\mathfrak{m}_p/\mathfrak{m}_p^{m+1},k\big)$.
\end{proof}

The jet space construction is a covariant functor:
If $f:X\to Y$ is a morphism of schemes over $k$ then we have the induced $\O_X$-algebra map $f^*\scoj^{(n)}_Y\to\scoj^{(n)}_X$ and hence a morphism of linear spaces over $X$, $\jet^n(X)\to \jet^n(Y)\times_YX$, which in turn induces $\jet^n(f):\jet^n(X)\to\jet^n(Y)$ over $f$.
At the level of $S$-points, under the identification given by Proposition~\ref{jetsatapoint},
$\jet^n(f)_p$ is the one induced by the natural map $p^{-1}(f^\sharp):f(p)^{-1}(\O_Y)\to p^{-1}\O_X$.
It is routine to check that if $f$ is a closed embedding then so is $\jet^n(f)$.

\begin{lemma}
\label{imagejet}
Let $f:X \hookrightarrow Y$ be a closed embedding of affine schemes over a field $k$.  
Let $p \in X(k)$ be a $k$-rational point.  Then the image
$\jet^n(f)\big(\jet^n(X)_p(k)\big)$ is
$$\big\{ \psi \in \hom_k(\fm_{Y,f(p)}/\fm_{Y,f(p)}^{n+1},k) : \psi (f) = 0 \text{ for all } f \in I(X)\cdot\big(\O_{Y,f(p)} /\fm_{Y,f(p)}^{n+1}\big)\big\}$$
\end{lemma}

\begin{proof}
Let $f^*:\O_{Y,f(p)}\to \O_{X,p}$ be the associated homomorphism on local rings.
Read through the identification of $\jet^n(X)_p(k)$ with $\hom_k(\fm_{X,p}/\fm_{X,p}^{n+1},k)$,
the map $\jet^n(f)$ is given by $\psi \mapsto \psi \circ f^*$.
Since $\ker(f^*)=I(X)\cdot\O_{Y,f(p)}$ and $\im(f^*)=\O_{X,p}$, this proves the lemma.
\end{proof}

\begin{corollary}
\label{algjetsdetermine}
Suppose $Z$ is an algebraic variety (i.e. a separated, integral scheme of finite type) over a field $k$.
If $X$ and $Y$ are irreducible closed subvarieties over $k$, and $p\in X(k)\cap Y(k)$ has the the property that $\jet^n(X)_p=\jet^n(Y)_p$ for all $n\in\mathbb{N}$, then $X=Y$.
\end{corollary}

\begin{proof}
If $U \subseteq Z$ is a dense open affine containing $p$
and $U \cap X = U \cap Y$, then $X = Y$.
Thus, we may assume that $Z$ is affine.
We show now that $I(X) \subseteq I(Y)$.  The opposite inclusion is 
shown by reversing the r\^{o}les of $X$ and $Y$.  If $f \in I(X)$, then by the description of 
the image of the $\jet^n(X)_p(k)$ in $\jet^n(Z)_p(k)$ from Lemma~\ref{imagejet}, every element of $\jet^n(X)_p(k) \leq \hom_k(\fm_{Z,p}/\fm_{Z,p}^{n+1},k)$ vanishes on $f$.  As $Z$ (and, hence, the local ring $\O_{Z,p}$) is noetherian, 
$\displaystyle \bigcap_{n \geq 0} \big(\fm_{Z,p}^{n+1} + I(Y) \cdot\O_{Z,p} \big) = I(Y) \cdot\O_{Z,p}$.
If $f \notin I(Y)$, then as $I(Y)$ is primary, $f \notin I(Y) \cdot\O_{Z,p}$.
So, for some $n$ we have $f \notin \fm_{Z,p}^{n+1} + I(Y) \cdot\O_{Z,p}$.
Again by the description of the image of $\jet^n(Y)_p(k)$, there would be some 
element of the image which did not vanish on $f$.
Since $\jet^n(X)_p(k)$
and $\jet^n(Y)_p(k)$ have the same image by assumption, this is be impossible.  
\end{proof}

\begin{lemma}
\label{separablejet}
Suppose $f:X \to Y$ is a dominant separable morphism of varieties over a feild $k$.
Then $\jet^n(f):\jet^n(X) \to\jet^n(Y)$ is a dominant morphism.
\end{lemma}

\begin{proof}
We may take $k=k^{\alg}$.
As $f$ is dominant and separable, there is a dense open $U \subseteq X$ on which 
$f$ is smooth.  That is, for every point $p \in U(k)$ the map $f^*:\fm_{Y,f(p)}/\fm_{Y,f(p)}^2 
\to \fm_{X,p}/\fm_{X,p}^2$ is injective.  It follows that for every $n>0$ that the 
map $f^*:\fm_{Y,f(p)}/\fm_{Y,f(p)}^{n+1} \to \fm_{X,p}/\fm_{X,p}^{n+1}$ is injective.  Hence, 
taking duals, $\jet^n(f)_p:\jet^n(X)_p(k) \to \jet^n(Y)_{f(p)}(k)$ is surjective.  As $f\big(U(k)\big)$ is 
dense in $Y$, we have that $\jet^n(f)\big(\jet^n(X)(k)\big)$ is Zariski-dense in $\jet^n(Y)$.
\end{proof}

\begin{lemma}
\label{etalejet}
Suppose $f:X \to Y$ is an \'etale morphism of schemes of finite type over a field $k$.
Then $\jet^n(f):\jet^n(X) \to\jet^n(Y)$ is \'etale.
Moreover, if $R$ is a $k$-algebra with $\spec(R)$ finite, and $p:\spec(R)\to X$ is an $R$-point of $X$ over $k$, then $\jet^n(f)_p:\jet^n(X)_p \to\jet^n(Y)_{f(p)}$ is an isomorphism.
\end{lemma}

\begin{proof}
Since $f$ is \'etale, for every topological point $q\in X$,
$f$ induces an isomorphism of finite dimensional $k$-vector spaces, $\mathfrak{m}_{f(q)}/\mathfrak{m}_{f(q)}^{m}\to\mathfrak{m}_q/\mathfrak{m}_q^m$, for all $m>0$.
Now for each $x\in\spec(R)$ we have the associated local $k$-algebra homomorphisms $p_x^\sharp:\O_{X,p(x)}\to R_{x}$.
Set $J_{p_x}:=\ker\big(p_x^\sharp\otimes\id:\O_{X,p(x)}\otimes_kR_{x}\to R_{x}\big)$.
Specialising Proposition~\ref{jetsatapoint} to the case of $S=\spec(R)$ we see that
$$p^*\scoj^{(n)}_X=\bigoplus_{x\in\spec(R)}J_{p_x}/J_{p_x}^{n+1}$$
and
$$f(p)^*\scoj^{(n)}_Y=\bigoplus_{x\in\spec(R)}J_{f(p)_x}/J_{f(p)_x}^{n+1}$$
Hence it suffices to show that for each $x\in\spec(R)$, $f$ induces an ismorphism between $J_{f(p)_x}/J_{f(p)_x}^{n+1}$ and $J_{p_x}/J_{p_x}^{n+1}$.
Since the maximal ideal $\mathfrak{m}_x$ in $R_{(x)}$ must be nilpotent, $J_{p_x}^{n+1}$ contains $\mathfrak{m}_{p(x)}^\ell\otimes 1$ and $J_{f(p)_x}^{n+1}$ contains $\mathfrak{m}_{f(p)(x)}^\ell\otimes 1$, for some $\ell> 0$.
But $f$ does induce an ismorphism $(\O_{X,p(x)}\otimes_kR_{x})/(\mathfrak{m}_{p(x)}^{\ell}\otimes 1)\to(\O_{Y,f(p)(x)}\otimes_kR_{x})/(\mathfrak{m}_{f(p)(x)}^{\ell}\otimes 1)$ which will take $J_{p_x}^{n+1}$ to $J_{f(p)(x)}^{n+1}$.
So $f$ induces an isomorphism between $f(p)^*\scoj^{(n)}_Y$ and $p^*\scoj^{(n)}_X$, and hence an isomorphism between $\jet^n(Y)_{f(p)}$ and $\jet^n(X)_p$, as desired.

The first part of the lemma now follows on general grounds.  To show that $\jet^n(f)$ is \'{e}tale we need to check 
that it is smooth and of relative dimension zero.  These properties are local.  As on the base $\jet^n(f)$
is simply $f$ which is \'{e}tale and $\jet^n(f)$ is an isomorphism fibrewise, $\jet^n(f)$ is of relative
dimension zero.
For smoothness consider the following diagram for any point $\tilde{p} \in \jet^n(X)(k^{\alg})$ lying above some point $p \in X(k^{\alg})$:
$$\xymatrix{
T_{\tilde{p}}\big(\jet^n(X)_p\big)\ar[d]\ar[rr]^{d_{\tilde{p}}\left(\jet^n(f)_p\right)\ \ \ } && T_{\jet^n(f)(\tilde{p})}\left(\jet^n(Y)_{f(p)}\right)\ar[d]\\
 T_{\tilde{p}}\jet^n(X)\ar[d]\ar[rr]^{d_{\tilde{p}}\jet^n(f)}  & & T_{\jet^n(f)(\tilde{p})}\jet^n(Y)\ar[d]\\
 T_pX\ar[rr]^{d_pf} && T_{f(p)}Y
  }$$
 As $\jet^n(f)$ restricts to isomorphism between
$\jet^n(X)_{p}$ and $\jet^n(Y)_{f(p)}$ we see that $d_{\tilde{p}}\left(\jet^n(f)_p\right)$ is an isomorphism between 
$T_{\tilde{p}} \left(\jet^n(X)_p\right)$ and $T_{\jet^n(f)(\tilde{p})}\left( \jet^n(Y)_{f(p)}\right)$.  As $f$ itself is 
\'{e}tale, $d_pf$ is an isomorphism between $T_p X$ and $T_{f(p)} Y$.  Hence, 
$d_{\tilde{p}}\jet^n(f)$ is an isomorphism and so $\jet^n(f)$ is \'{e}tale. 
\end{proof}

\begin{remark}
\label{restrictopenjet}
It follows from Lemma~\ref{etalejet} that if $U$ is a Zariski open subset of $X$ then $\jet^nU\to U$ is the restriction of $\jet^nX\to X$ to $U$.
\end{remark}

\bigskip

\subsection{A co-ordinate description of the jet space}
Let $X\subset \mathbb{A}_k^\ell$ be an affine scheme of finite type over a ring $k$.
We wish to give a co-ordinate description of $\jet^n(X)$ as a subscheme of $\jet^n(\mathbb{A}_k^\ell)$.

If $x=(x_1,\dots,x_\ell)$ are co-ordinates for $\mathbb{A}_k^\ell$ then $\Gamma(\mathbb{A}_k^\ell,\scoj^{(n)}_{\mathbb{A}_k^\ell})=J/J^{n+1}$
where $J$ is the ideal in $k[x,x']$ generated by elements of the form $z_i:=(x_i'-x_i)$.
Setting $z=(z_1,\dots,z_\ell)$ we have that $z=x'-x$.
Now let
$\displaystyle \Lambda:= \{\alpha\in \mathbb{N}^\ell : 0<\sum_{i=1}^\ell\alpha_i\leq n\}$.
We use multi-index notation so that for each $\alpha\in\Lambda$,
$\displaystyle z^\alpha:=\prod_{i=1}^\ell z_i^{\alpha_i}$.
Note that $J/J^{n+1}$ is freely generated as a $k[x]$-module by
$\{z^\alpha(\mod J^{n+1}):\alpha\in\Lambda\}$.
That is, $\jet^n(\mathbb{A}_k^\ell)$ is the affine space $\spec\big(k[x,(z_\alpha)_{\alpha\in\Lambda}]\big)$.

Suppose $X=\spec\big(k[x]/I\big)$.
Now $\Gamma(X,\scoj_X^{(n)})=\langle JII'\rangle/\langle J^{n+1}II'\rangle$ where $I'$ is just $I$ in the indeterminates $x'$.
As a $k[x]/I$-module $\Gamma(X,\scoj_X^{(n)})$ is generated by the image of
$\{ z^\alpha(\mod J^{n+1}):\alpha\in\Lambda\}$.
By the construction of linear spaces, to describe $\jet^n(X)$ we need to describe the relations among these generators.
The relations are obtained by writing $P(x')$, for each $P\in I$, as a $k[x]$-linear combination of $\{ z^\alpha(\mod J^{n+1}):\alpha\in\Lambda\}$ in $k[x,x']/J^{n+1}$.

To this end, for each $\alpha\in\mathbb{N}^\ell$ let $D^\alpha$ be the differential operator on $k[x]$ given by
$$D^\alpha\big(\sum_{\beta\in B}r_\beta x^\beta\big)=\sum_{\beta\in B,\beta\geq\alpha}r_\beta\binom{\beta}{\alpha}x^{\beta-\alpha}.$$
That is, for any $P\in k[x]$,
$\displaystyle P(x+z)=\sum_\alpha (D^\alpha P)(x)z^\alpha$.
Note that if $\alpha!:=\alpha_1!\cdots\alpha_\ell!$ is invertible in $k$ then $D^\alpha$ is just the differential operator
$\displaystyle\frac{1}{\alpha!}\frac{\partial^{|\alpha|}}{\ \partial x_1^{\alpha_1}\cdots\partial x_\ell^{\alpha_\ell}}$.
Now consider $\displaystyle P(x)=\sum_{\beta\in B}r_\beta x^\beta \in I$.
Then
\begin{eqnarray*}
P(x') & = & P(x+z)\\
& = & \sum_\alpha (D^\alpha P)(x)z^\alpha\\
& = & \sum_{\alpha\in\Lambda}(D^\alpha P)(x)z^\alpha \ \ \mod J^{n+1}
\end{eqnarray*}
We have thus shown:

\begin{proposition}
\label{jet-coordinate}
With the above notation,
if $X=\spec(k[x]/I)$ then as a sub-scheme of the affine space $\jet^n(\mathbb{A}_k^\ell)=\spec\big(k[x,(z_\alpha)_{\alpha\in\Lambda}]\big)$, $\jet^n(X)$ is given by the equations:
\begin{eqnarray*}
P(x) & = &0\\
\sum_{\alpha\in\Lambda}(D^\alpha P)(x)z_\alpha & = & 0
\end{eqnarray*}
for each $P\in I$.
\end{proposition}

\begin{remark}
The co-ordinate description of the jet space given by Proposition~\ref{jet-coordinate} agrees with the way  that jet spaces of algebraic varieties are {\em defined} in~\cite{pillayziegler03}.
\end{remark}

\bigskip

\section{Interpolation}
\label{sectinterpolating}

\noindent
In this section we introduce a natural map that compares the jet space of a prolongation to the prolongation of the jet space.
This morphism will allow us, in a sequel paper, to define the Hasse-differential jet spaces of Hasse-differential varieties.

Fix a finite free $\mathbb{S}$-algebra scheme $\E$ over a ring $A$, an $A$-algebra $k$, an $A$-algebra homomorphism $e:k\to\E(k)$, and a scheme $X$ over $k$.
Let $\tau(X)=\tau(X,\E,e)$ be the prolongation of $X$ with respect to $\E$ and $e$. (Recall that our standing assumption is that this prolongation space exists, which is the case for example when $X$ is quasi-projective.)
Fix also $m\in\mathbb{N}$.
We construct a map
$$\phi_{m,\E,e}^X:\jet^m \big(\tau(X)\big) \to \tau\big(\jet^m (X)\big)$$
over $X$, which we will call the {\em interpolating map} of $X$ (with respect to $m$, $\E$, and $e$).
We will define $\phi_{m,\E,e}^X$ by expressing its action on the $R$-points of $\jet^m\tau(X)$, for arbitrary $k$-algebras $R$.
It should be clear from the construction, and will also follow from the co-ordinate description given in the next section, that $\phi_{m,\E,e}^X$ is a morphism of schemes over $k$.

Our map will be the composition of two other maps which we now describe.
Suppose $p:\spec(R)\to \tau(X)$ is an $R$-point of $\tau(X)$ over $k$.
Then $p\times_k\E(k):\spec\big(\E(R)\big)\to\tau(X)\times_k\E(k)$ is an $\E(R)$-point of $\tau(X)\times_k\E(k)$ over $\E(k)$.

\begin{lemma}
\label{basechangemap}
Base change from $k$ to $\E(k)$ induces a natural map
$$u:=u^{X,p}_{m,\E,e}:\jet^m\big(\tau(X)\big)_p(R)\longrightarrow\jet^m\big(\tau(X)\times_k\E(k)\big)_{p\times_k\E(k)}\big(\E(R)\big)$$
\end{lemma}

\begin{proof}
Recall that
$\jet^m\big(\tau(X)\big)_p(R)
=
\hom_R(p^*\scoj^{(m)}_{\tau(X)},R)$
and
\begin{eqnarray*}
\jet^m\big(\tau(X)\times_k\E(k)\big)_{p\times_k\E(k)}\big(\E(R)\big)
&=&
\hom_{\E(R)}\big((p\times_k\E(k))^*(\scoj^{(m)}_{\tau(X)\times_k\E(k)}),\E(R)\big)\\
&=&
\hom_{\E(R)}\big(p^*\scoj^{(m)}_{\tau(X)}\otimes_k\E(k),\E(R)\big)
\end{eqnarray*}
where the identification
$(p\times_k\E(k))^*(\scoj^{(m)}_{\tau(X)\times_k\E(k)})=p^*\scoj^{(m)}_{\tau(X)}\otimes_k\E(k)$ is by the fact that $\scoj^{(m)}_{\tau(X)\times_k\E(k)}$ is just the pull back of $\scoj^{(m)}_{\tau(X)}$ under the projection $\tau(X)\times_k\E(k)\to\tau(X)$.

Now define $u$ to be the map that assigns to the $R$-linear map $\nu:p^*\scoj_{\tau(X)}^{(m)}\to R$ the $\E(R)$-linear map $\nu\otimes_k\E(k): p^*\scoj_{\tau(X)}^{(m)}\otimes_k\E(k)\to \E(R)$.
That is, $u$ is given by base change.
\end{proof}

Under the usual identification $p$ corresponds to an $\E^e(R)$-point of $X$ over $k$, $\hat p:\spec\big(\E^e(R)\big)\to X$.

\begin{lemma}
\label{weilmap}
Applying the $\jet^m$ functor to $r^X_{\E,e}:\tau(X)\times_k\E(k)\to X$, the canonical morphism associated to $\tau(X)$, induces a map
$$v:=v^{X,p}_{m,\E,e}:\jet^m\big(\tau(X)\times_k\E(k)\big)_{p\times_k\E(k)}\big(\E(R)\big)\longrightarrow\jet^{m}(X)_{\hat p}\big(\E^e(R)\big)$$
\end{lemma}

\begin{proof}
Note that as $\E(k)$-algebras $\E(R)=\E^e(R)$.
So $p\times_k\E(k)$ can also be viewed as an $\E^e(R)$-point of $\tau(X)\times_k\E(k)$ over $\E(k)$.

Applying the jet functor we get
$\jet^m(r^X_{\E,e}):\jet^m\big(\tau(X)\times_k\E(k)\big)\to\jet^m(X)$.
Since $\hat p=r^X_{\E,e}\circ\big(p\times_k\E(k)\big)$ -- see Lemma~\ref{taunpoints} -- this morphism restricted to the fibre at the $\E^e(R)$-point $p\times_k\E(k)$ yields a morphism
$$\jet^m\big(\tau(X)\times_k\E(k)\big)_{p\times_k\E(k)}\longrightarrow\jet^{m}(X)_{\hat p}$$
Evaluating at $\E^e(R)$-points yields the desired map $v$.
\end{proof}

Our interpolating map is now just the composition of the maps given in the above two lemmas.
More precisely, if $p\in\tau(X)(R)$ and $\nu\in\jet^m\big(\tau(X)\big)_p(R)$ then we define our interpolating map by
$$\phi^X_{m,\E,e}(p,\nu):=\big(\hat p,v(u(\nu))\big)$$
where $u$ is from Lemma~\ref{basechangemap} and $v$ is from Lemma~\ref{weilmap}.
Note that
$$\phi^X_{m,\E,e}(p,\nu)\in\jet^m(X)\big(\E^e(R)\big)=\tau\big(\jet^m(X)\big)(R)$$
Depending on what we wish to emphasise/suppress, we may drop one or more of the subscripts and superscripts on $\phi^X_{m,\E,e}$.

\begin{lemma}
\label{interpislin}
The interpolating map $\phi^X_{m,\E,e}:\jet^m\tau(X)\to\tau\jet^m(X)$ is a morphism of linear spaces over $\tau(X)$.
\end{lemma}

\begin{proof}
That $\phi$ is a morphism of schemes over $\tau(X)$ can be derived from the definition, but also follows from the co-ordinate description given in the next section.

First note that the prolongation functor preserves products and that $\tau(\mathbb{S}_k)=\E_k$.
Hence, it takes the $\mathbb{S}_k$-linear space $\jet^m(X)\to X$ to an $\E_k$-linear space $\tau\big(\jet^m(X)\big)\to\tau(X)$.
The latter obtains an $\mathbb{S}_k$-linear space structure from $s:\mathbb{S}\to\E$.
It is with respect to this structure that the lemma is claiming $\phi$ is linear.

It is clear from the definition of $u$ and $v$ in Lemmas~\ref{basechangemap} and~\ref{weilmap} respectively, that for an arbitrary $k$-algebra $R$, and an arbitrary $R$-point $p$ of $\tau(X)$, $u$ and $v$ are $R$-linear.
Hence $\phi$ is $R$-linear on the $R$-points of the fibre above $p$.
As $R$ and $p$ were arbitrary, this implies that $\phi$ is a morphism of linear spaces.
\end{proof}

The fundamental properties of the interpolating map are given in the following proposition.

\begin{proposition}
\label{phi-properties}
The interpolating map satisfies the following properties.
\begin{itemize}
\item[(a)]
{\em Functoriality.}
If $g:X\to Y$ is a morphism of schemes over $k$, then for each $m,n\in\mathbb{N}$ the following diagram commutes
$$\xymatrix{
\jet^m\tau(X)\ar^{\jet^m\tau(g)}[rr]\ar[d]^{\phi^X} && \jet^m\tau(Y)\ar^{\phi^Y}[d] \\
\tau\jet^m(X)\ar_{\tau\jet^m(g)}[rr] && \tau\jet^m(Y)
}$$
\item[(b)]
{\em Compatibility with composition of prolongations.}
Suppose $\F$ is another finite free $\mathbb{S}$-algebra scheme and $f:k\to\F(k)$ is a ring homomorphism agreeing with $s_\F^k$ on $A$.
Then the following commutes
$$\xymatrix{
\jet^m\tau(X,\E\F,ef)\ar[dd]^{\phi^X_{\E\F,ef}}\ar[drr]^{\phi^{\tau(X,\E,e)}_{\F,f}}\\
&&\tau\big(\jet^m\tau(X,\E,e),\F,f\big)\ar[dll]^{\tau(\phi^X_{\E,e},\F,f)}\\
\tau\big(\jet^m(X),\E\F,ef\big)
}$$
\item[(c)]
{\em Compatibility with comparing of prolongations.}
Suppuse $\F$ is another finite free $\mathbb{S}$-algebra scheme, $f:k\to\F(k)$ is a ring homomorphism agreeing with $s_\F^k$ on $A$, and $\alpha:\E\to\F$ is a morphisms of ring schemes such that $\alpha^k\circ e=f$.
Then the following diagram commutes:
$$\xymatrix{
\jet^m\tau(X,\E,e)\ar[d]^{\phi_{\E}}\ar[rr]^{\jet^m(\hat\alpha)}&&\jet^m\tau(X,\F,f)\ar[d]^{\phi_\F}\\
\tau\big(\jet^m(X),\E,e\big)\ar[rr]_{\hat\alpha^{\jet^m(X)}}&&\tau\big(\jet^m(X),\F,f\big)
}$$
\end{itemize}
\end{proposition}

\begin{proof}
For part~(a), fix $p\in\tau(X)(R)$.
By the funtoriality of the Weil restriction (see Proposition~\ref{weil}) it follows that the following diagram commutes:
$$\xymatrix{
\tau(X)\times_k\E(k)\ar[rr]^{\tau(g)\times_k\E(k)}\ar[d]^{r^X_{\E,e}} && \tau(Y)\times_k\E(k)\ar[d]^{r^Y_{\E,e}}\\
X\ar[rr]^g && Y
}$$
Taking jets and evaluating at $\E^e(R)$-points we get
$$\xymatrix{
\jet^m\big(\tau(X)\times_k\E(k)\big)_{p\times_k\E(k)}\big(\E^e(R)\big)\ar[rr]^{\jet^m\big(\tau(g)\times_k\E(k)\big)}\ar[d]^{v^{X,p}} && \jet^m\big(\tau(Y)\times_k\E(k)\big)_{\tau(g)(p)\times_k\E(k)}\big(\E^e(R)\big)\ar[d]^{v^{Y,\tau(g)(p)}}\\
\jet^m(X)_{\hat p}\big(\E^e(R)\big)\ar[rr]^{\jet^m(g)} && \jet^m(Y)_{\widehat{\tau(g)(p)}}\big(\E^e(R)\big)
}$$
On the other hand,
that the following diagram commutes is clear from the fact that the map $u$ in Lemma~\ref{basechangemap} is just given by base change:
$$\xymatrix{
\jet^m\big(\tau(X)\big)_p(R)\ar[d]^{u^{X,p}}\ar[rr]^{\jet^m\big(\tau(g)\big)} && \jet^m\big(\tau(Y)\big)_{\tau(g)(p)}(R)\ar[d]^{u^{Y,\tau(g)(p)}}\\
\jet^m\big(\tau(X)\times_k\E(k)\big)_{p\times_k\E(k)}\big(\E(R)\big)\ar[rr]^{\jet^m\big(\tau(g)\times_k\E(k)\big)} && \jet^m\big(\tau(Y)\big)_{\tau(g)(p)\times_k\E(k)}\big(\E(R)\big)
}$$
Putting these two commuting squares together yeilds
$$\xymatrix{
\jet^m\big(\tau(X)\big)_p(R)\ar[d]^{v^{X,p}\circ u^{X,p}}\ar[rr]^{\jet^m\big(\tau(g)\big)} && \jet^m\big(\tau(Y)\big)_{\tau(g)(p)}(R)\ar[d]^{v^{Y,\tau(g)(p)}\circ u^{Y,\tau(g)(p)}}\\
\jet^m(X)_{\hat p}\big(\E^e(R)\big)\ar[rr]^{\jet^m(g)} && \jet^m(Y)_{\widehat{\tau(g)(p)}}\big(\E^e(R)\big)
}$$
By the construction of the interpolating map, this in turn implies 
$$\xymatrix{
\jet^m\tau(X)\ar^{\jet^m\tau(g)}[rr]\ar[d]^{\phi^X} && \jet^m\tau(Y)\ar^{\phi^Y}[d] \\
\tau\jet^m(X)\ar_{\tau\jet^m(g)}[rr] && \tau\jet^m(Y)
}$$
as desired.

For part~(b), 
let us make the systematic abbreviation $\tau^\E(X)$ for $\tau(X,\E,e)$.
Fixing $p\in \tau^{\E\F}(X)(R)$ we have the associated points $\hat p\in X\big(\E\F^{ef}(R)\big)$ and $\hat p_\F\in\tau^\E(X)\big(\F^f(R)\big)$.
The contribution to the interpolating map from base change (namely from the map given by Lemma~\ref{basechangemap}) will cause no difficulty and so we only check commuting for the relevant diagram coming from the Weil restriction map (i.e., the map from Lemma~\ref{weilmap}).
From Lemma~\ref{rcomposed} we have the following commuting triangle
$$\xymatrix{
\tau^{\E\F}(X)\times_k\E\F(k)\ar[dd]_{r^{\tau^\E(X)}_{\F,f}\times_k\E(k)}\ar[drr]^{r^X_{\E\F,ef}}\\
&&X\\
\tau^\E(X)\times_k\E(k)\ar[urr]_{r^X_{\E,e}}
}$$
This induces the following morphism of sheaves of $\E\F(R)$-algebras
$$\xymatrix{
p^*(\scoj_{\tau^{\E\F}(X)}^{(m)})\otimes_k\E\F(k)\\
&&\hat p^*(\scoj_X^{(m)})\ar[ull]\ar[dll]\\
\hat p_\F^*(\scoj_{\tau^\E(X)}^{(m)})\otimes_k\E(k)\ar[uu]
}$$
Taking duals, and making natural identifications we get:
$$\xymatrix{
\hom_{\E\F(R)}\big(p^*(\scoj_{\tau^{\E\F}(X)}^{(m)})\otimes_k\E\F(k),\E\F(R)\big)\ar[drr]^{v^{X,p}_{\E\F}}\ar[dd]^{v^{\tau^\E(X),p}_\F\otimes_k\E(k)}\\
&&\hom_{\E\F^{ef}(R)}\big(\hat p^*(\scoj_X^{(m)}),\E\F^{ef}(R)\big)\\
\hom_{\E\big(\F^{f}(R)\big)}\big(\hat p_\F^*(\scoj_{\tau^\E(X)}^{(m)})\otimes_k\E(k),\E(\F^{f}(R))\big)\ar[urr]^{v^{X,\hat p_\F}_\E}
}$$
From here part~(b) is easily verified.

For part~(c) we continue to use the abbreviation $\tau^\E(X)$ for $\tau(X,\E,e)$, and we work with a fixed point $p\in \tau^{\E}(X)(R)$ for some fixed $k$-algebra $R$.
Again we are going to break the desired commutative diagram into two peices, one coming from each of the two ingredients of the interpolating map (namely from the map given by Lemma~\ref{basechangemap} and the map from Lemma~\ref{weilmap}).
To do so, note first of all that there is a natural map
$$\tilde{\alpha}:\jet^m\big(\tau^{\E}(X)\times_k\E(k)\big)_{p\times_k\E(k)}\big(\E(R)\big)\to\jet^m\big(\tau^{\F}(X)\times_k\F(k)\big)_{\hat\alpha(p)\times_k\F(k)}\big(\F(R)\big).$$
Indeed, $\tilde{\alpha}$ is just the composition of the map
$$\jet^m\big(\tau^{\E}(X)\times_k\E(k)\big)_{p\times_k\E(k)}\big(\E(R)\big)\to
\jet^m\big(\tau^{\E}(X)\times_k\F(k)\big)_{p\times_k\F(k)}\big(\F(R)\big)$$
induced by base change from $\E(k)$ to $\F(k)$ using $\alpha$, with
$\jet^m\big(\hat\alpha\times_k\F(k)\big):$
$$\jet^m\big(\tau^{\E}(X)\times_k\F(k)\big)_{p\times_k\F(k)}\big(\F(R)\big)\to
\jet^m\big(\tau^{\F}(X)\times_k\F(k)\big)_{\hat\alpha(p)\times_k\F(k)}\big(\F(R)\big).$$
Hence, to obtain the desired commuting diagram it will suffice to show
\begin{itemize}
\item[(1)]
The following diagram commutes:
$$\xymatrix{
\jet^m\big(\tau^{\E}(X)\big)_p(R)\ar[d]\ar[rr]^{\jet^m(\hat\alpha)} && \jet^m\big(\tau^{\F}(X)\big)_{\hat\alpha(p)}\ar[d]\\
\jet^m\big(\tau^{\E}(X)\times_k\E(k)\big)_{p\times_k\E(k)}\big(\E(R)\big)\ar[rr]^{\tilde{\alpha}} && \jet^m\big(\tau^{\F}(X)\times_k\F(k)\big)_{\hat\alpha(p)\times_k\F(k)}\big(\F(R)\big)
}$$
where the vertical arrows are the base change maps of Lemma~\ref{basechangemap}, and
\item[(2)]
The following diagram commutes:
$$\xymatrix{
\jet^m\big(\tau^{\E}(X)\times_k\E(k)\big)_{p\times_k\E(k)}\big(\E(R)\big)\ar[d]\ar[rr]^{\tilde{\alpha}} && \jet^m\big(\tau^{\F}(X)\times_k\F(k)\big)_{\hat\alpha(p)\times_k\F(k)}\big(\F(R)\big)\ar[d]\\
\jet^m(X)_{\hat p}\big(\E^e(R)\big)\ar[rr]^{\hat\alpha^{\jet^m(X)}} &&
\jet^m(X)_{\widehat{\hat\alpha(p)}}\big(\F^f(R)\big)
}$$
where the vertical arrows are the maps of Lemma~\ref{weilmap}.
\end{itemize}
Diagram~(1) is easily seen to commute by unravelling the definitions and using that fact that jet spaces commute with base change.
So we focus on proving that diagram~(2) commutes.

From lemma~\ref{rcompared} we have the following commuting diagram
$$\xymatrix{
\tau^{\F}(X)\times_k\F(k)\ar[dr]_{r^X_{\F,f}} && \tau^{\E}(X)\times_k\F(k)\ar[dl]^{\ \ r^X_{\E,e}\circ(\id\otimes\underline{\alpha})}\ar[ll]_{\hat\alpha\otimes\id_{\F(k)}}\\
&X
}$$
where $r^X_{\E,e}\circ(\id\otimes\underline{\alpha}):\tau^{\E}(X)\times_k\F(k)\to X$ is the composition of the projection $\tau^{\E}(X)\times_k\F(k)=\tau^{\E}(X)\times_k\E(k)\times_{\E(k)}\F(k)\to \tau^{\E}(X)\times_k\E(k)$ with $r^X_{\E,e}:\tau^{\E}(X)\times_k\E(k)\to X$.
Taking the co-jet sheaves and pulling back by the appropriate morphisms we obtain the following commuting diagram of sheaves on $\spec\big(\F(R)\big)$:
$$\xymatrix{
\big(\hat\alpha(p)\times_k\F(k)\big)^*\scoj_{\tau^{\F}(X)\times_k\F(k)}^m\ar[rr]^c &&
\big(p\times_k\F(k)\big)^*\scoj_{\tau^{\E}(X)\times_k\F(k)}^m\\
&\widehat{\hat\alpha(p)}^*\scoj_X^m\ar[ul]^b\ar[ur]_a
}$$
Note that $r^X_{\E,e}\circ(\id\otimes\underline{\alpha})$ really does take $p\times_k\F(k)$ to $\widehat{\hat\alpha(p)}$ since $\widehat{\hat\alpha(p)}=\hat p\circ\underline{\alpha}$.

Now, let us consider diagram~(2) above.
Unraveling the definitions it is not hard to verify that given
$\gamma\in\jet^m\big(\tau^{\E}(X)\times_k\E(k)\big)_{p\times_k\E(k)}\big(\E(R)\big)$
\begin{itemize}
\item
going clockwise along diagram~(2) takes $\gamma$ to $(\gamma\otimes_{\E(k)}\F(k)\big)\circ c\circ b$; while
\item
going counter-clockwise along diagram~(2) takes $\gamma$ to $(\gamma\otimes_{\E(k)}\F(k)\big)\circ a$.
\end{itemize}
Hence diagram~(2) is commutative, as desired.
\end{proof}

\bigskip
\subsection{The interpolating map in co-ordinates}
We wish to give a co-ordinate description of the interpolating map for affine schemes of finite type.
Since the jet and prolongation functors preserve closed embeddings,
functoriality allows us to reduce this task to affine space.

Fix a finite free $\mathbb{S}$-algebra scheme $\E$ over a ring $A$ and an $A$-algebra $k$ equipped with an $A$-algebra homomorphism $e:k\to\E(k)$.
Let $(e_0=1,e_1,\dots,e_{\ell-1})$ be a basis for $\E(k)$ over $k$.

Consider $\mathbb{A}_k^r=\spec\big(k[x_1,\dots,x_r]\big)$.
Then $\tau\mathbb{A}_k^r=\spec\big(k[\overline{y}_1,\dots,\overline{y}_r]\big)$, where each
$\overline{y}_i=(y_{i,0},\dots,y_{i,\ell-1})$.
Let $x=(x_1,\dots,x_r)$ and $\overline{y}=(\overline{y}_1,\dots,\overline{y}_r)$.

Suppose $R$ is a $k$-algebra, $\overline{a}\in\tau\mathbb{A}_k^r(R)$ and $a\in \mathbb{A}_k^r\big(\E^e(R)\big)$ is the point corresponding to $\overline{a}$.
A straightforward computation using Proposition~\ref{jetsatapoint} shows that
\begin{eqnarray*}
\jet^m(\tau\mathbb{A}_k^r)_{\overline{a}}(R)
& =&
\hom_R\big(R[\overline{y}]_{\overline{a}}/(\overline{y}-\overline{a})^{m+1},R\big)\\
\jet^m(\mathbb{A}_k^r)_a\big(\E^e(R)\big)
&=&
\hom_{\E^e(R)}\big(\E^e(R)[x]_a/(x-a)^{m+1},\E^e(R)\big)
\end{eqnarray*}
where $R[\overline{y}]_{\overline{a}}$ is the localisation of $R[\overline{y}]$ at $\{f\in R[\overline{y}]:f(\overline{a})\in R^\times\}$, and $\E^e(R)[x]_a$ is the localisation of $\E^e(R)[x]$ at $\{f\in \E^e(R)[x]:f(a)\in \E^e(R)^\times\}$.
Our interpolating map
$$\hom_R\big(R[\overline{y}]_{\overline{a}}/(\overline{y}-\overline{a})^{m+1},R\big)\longrightarrow \hom_{\E^e(R)}\big(\E^e(R)[x]_a/(x-a)^{m+1},\E^e(R)\big)$$
is given by $f\longmapsto \big(f\times_k\E(k)\big)\circ r^*$ where
$$r^*:\E^e(R)[x]_a/(x-a)^{m+1}\longrightarrow \E(R)[\overline{y}]_{\overline{a}}/(\overline{y}-\overline{a})^{m+1}$$
is the map induced by $\displaystyle x_i\longmapsto \sum_{j=0}^{\ell-1}y_{i,j}e_j$.
So to compute the interpolating map on co-ordinates, we need to compute $r^*$ on the monomial basis for $\E^e(R)[x]_a/(x-a)^{m+1}$ over $\E(R)$.
To that end, fix $\beta\in\mathbb{N}^r$ with $0<| \beta |\leq m$ and compute in multi-index notation
\begin{eqnarray*}
\label{finitial}
r^*(x^\beta)
& = &
\prod_{i=1}^r\big(\sum_{j=0}^{\ell-1}y_{i,j}e_j\big)^{\beta_i}\\
& = &
\sum_{\gamma=(\gamma_1,\dots,\gamma_r)\in\mathbb{N}^{\ell r}, |\gamma_i|=\beta_i}\overline{y}^{\gamma}(e_0,\dots,e_{\ell-1})^{\sum_{i=1}^r\gamma_i}
\end{eqnarray*}
Set $\Gamma_\beta:=\{\gamma=(\gamma_1,\dots,\gamma_r)\in\mathbb{N}^{\ell r} \ : \  |\gamma_i|=\beta_i\}$ and for each $\gamma\in\Gamma_\beta$, expand
\begin{eqnarray}
\label{expand}
(e_0,\dots,e_{\ell-1})^{\sum_{i=1}^r\gamma_i}=\sum_{j=0}^{\ell-1}c_{\gamma,j}e_j
\end{eqnarray}
for some $c_{\gamma,j}\in A$.
So we have
\begin{eqnarray}
\label{ronbeta}
r^*(x^\beta) & = & \sum_{j=0}^{\ell-1}\big(\sum_{\gamma\in\Gamma_\beta} c_{\gamma,j}\overline{y}^\gamma\big) e_j
\end{eqnarray}

\begin{remark}
\label{simplify}
Fix $\beta\in\mathbb{N}^r$ with $0<|\beta|\leq m$.
Let
$$\hat\beta:=(\beta_1,0,\dots;\beta_2,0,\dots;\dots;\beta_r,0,\dots,0)\in\Gamma_\beta\subset \mathbb{N}^{\ell r}.$$
Then
\begin{itemize}
\item
$c_{\hat\beta,0}=1$,
\item
$c_{\hat\beta,j}=0$ for all $j\neq 0$, and
\item
$c_{\gamma,0}=0$ for all $\gamma\in\Gamma_\beta\setminus \{\hat\beta\}$.
\end{itemize}
Indeed, this is because $e_0=1$ and if $\gamma\in\Gamma_\beta\setminus\{\hat\beta\}$ then $\displaystyle (e_0,\dots,e_{\ell-1})^{\sum_{i=1}^r\gamma_i}$ is in the kernel of the reduction map $\E(k)\to k$.
\end{remark}

We can already prove the following surjectivity result:

\begin{proposition}
\label{affinesurjective}
The interpolating map on affine space, $\phi:\jet^m\tau\mathbb{A}_k^r\to\tau\jet^m\mathbb{A}_k^r$, is surjective.
\end{proposition}

\begin{proof}
Fix a $k$-algebra $R$, a point $\overline{a}\in\tau\mathbb{A}_k^r(R)$, and let $a\in \mathbb{A}_k^r\big(\E^e(R)\big)$ be the point corresponding to $\overline{a}$, as above.
We need to show that
$$\hom_R\big(R[\overline{y}]_{\overline{a}}/(\overline{y}-\overline{a})^{m+1},R\big)\longrightarrow \hom_{\E^e(R)}\big(\E^e(R)[x]_a/(x-a)^{m+1},\E^e(R)\big)$$
given by $f\longmapsto \big(f\times_k\E(k)\big)\circ r^*$ is surjective.
Fix an arbitrary $\beta\in\mathbb{N}^r$ with $0<|\beta|\leq m$ and let $\nu\in\hom_{\E^e(R)}\big(\E^e(R)[x]_a/(x-a)^{m+1},\E^e(R)\big)$ be such that it takes $x^\beta$ to $1$ and all other monomial basis elements to $0$.
It suffices to show that $\nu$ is in the image of the interpolating map.
Define $f\in \hom_R\big(R[\overline{y}]_{\overline{a}}/(\overline{y}-\overline{a})^{m+1},R\big)$ such that $f(\overline{y}^{\hat\beta})=1$ and $f$ sends all other monomial basis elements to $0$.
Then
\begin{eqnarray*}
\big(f\times_k\E(k)\big)\circ r^*(x^\beta)
&=&
\sum_{j=0}^{\ell-1}f\big(\sum_{\gamma\in\Gamma_\beta} c_{\gamma,j}\overline{y}^\gamma\big) e_j\\
&=&
\sum_{j=0}^{\ell-1}c_{\hat\beta,j}e_j\\
&=&
e_0\\
 &=& 1
\end{eqnarray*}
where the first equality is by~(\ref{ronbeta}) and the penultimate equality is by Remark~\ref{simplify}.
On the other hand, for $\beta'\neq\beta$, since $\hat\beta\notin\Gamma_{\beta'}$, a similar calculation shows that $\big(f\times_k\E(k)\big)\circ r^*(x^{\beta'})=0$.
So $\big(f\times_k\E(k)\big)\circ r^*=\nu$, as desired.
\end{proof}

We return to the computation of the interpolating map on co-ordinates.
Fix $\gamma=(\gamma_1,\dots,\gamma_r)\in\mathbb{N}^{\ell r}$ such that $0<|\gamma|\leq m$ and consider the basis element $f_\gamma\in \hom_R\big(R[\overline{y}]_{\overline{a}}/(\overline{y}-\overline{a})^{m+1},R\big)$ which takes $\overline{y}^{\gamma}$ to $1$ and sends all other monomial basis elements to $0$.
We now compute what the interpolating map does to $f_\gamma$.
Fix $\beta\in\mathbb{N}^r$ with $0<|\beta|\leq m$.

\begin{eqnarray*}
\big(f_\gamma\times_k\E(k)\big)\circ r^*(x^\beta)
&=&
\sum_{j=0}^{\ell-1}f_\gamma\big(\sum_{\gamma'\in\Gamma_\beta} c_{\gamma,j}\overline{y}^{\gamma'}\big) e_j\\
&=&
\left\{ \begin{array}{ll}
\sum_{j=0}^{\ell-1}c_{\gamma,j}e_j & \mbox{if $\gamma\in\Gamma_\beta$};\\
        0 & \mbox{otherwise}.\end{array} \right.
\end{eqnarray*}
where the first equality is by~(\ref{ronbeta}) and the coefficients $c_{\gamma,j}\in A$ are from~(\ref{expand}).
Set $\tilde\gamma=(|\gamma_1|,\dots,|\gamma_r|)\in\mathbb{N}^r$.
So $\tilde\gamma$ is the multi-index such that $\gamma\in\Gamma_{\tilde\gamma}$.
We have shown that
\begin{eqnarray*}
\phi(f_\gamma)
&=&
\big(\sum_{j=0}^{\ell-1}c_{\gamma,j}e_j\big)g_{\tilde\gamma}
\end{eqnarray*}
where $g_{\tilde\gamma}\in\hom_{\E^e(R)}\big(\E^e(R)[x]_a/(x-a)^{m+1},\E^e(R)\big)$ takes $x^{\tilde\gamma}$ to $1$ and all other monomial basis elements to $0$.

For an arbitrary element  $\displaystyle f=\sum_{\gamma\in\mathbb{N}^{\ell r}, 0<|\gamma|\leq m}u_\gamma f_\gamma$ of $\hom_R\big(R[\overline{y}]_{\overline{a}}/(\overline{y}-\overline{a})^{m+1},R\big)$, where the $u_\alpha\in R$, we can then compute
\begin{eqnarray*}
\label{phionmu}
\phi(f)
& = &
\sum_{\gamma\in\mathbb{N}^{\ell r}, 0<|\gamma|\leq m}u_\gamma\big(\sum_{j=0}^{\ell-1}c_{\gamma,j}e_jg_{\tilde\gamma}\big)\\
&=&
\sum_{\beta\in\mathbb{N}^r,0<|\beta|\leq m}
\left(\sum_{j=0}^{\ell-1}\big(\sum_{\gamma\in\Gamma_\beta}u_\gamma c_{\gamma,j}\big)e_j\right)g_\beta
\end{eqnarray*}
We have shown:

\begin{proposition}
\label{phicoord}
If $R$ is a $k$-algebra and $(\overline{a},f)\in\jet^m\tau\mathbb{A}_k^r(R)$, where $\displaystyle f=\sum_{\gamma\in\mathbb{N}^{\ell r}, 0<|\gamma|\leq m}u_\gamma f_\gamma$, then $\phi(\overline{a},f)\in\tau\jet^m\mathbb{A}_k^r(R)$ is given by
$$\left(
\overline{a},
u_{\hat\beta},
\big(\sum_{\gamma\in\Gamma_\beta}u_\gamma c_{\gamma,1}\big),
\big(\sum_{\gamma\in\Gamma_\beta}u_\gamma c_{\gamma,2}\big),\dots,
\big(\sum_{\gamma\in\Gamma_\beta}u_\gamma c_{\gamma,\ell-1}\big)
\right)_{\beta\in\mathbb{N}^r,0<|\beta|\leq m}
$$
where the coefficients $c_{\gamma,j}\in A$ come from~(\ref{ronbeta}) above.
Stated another way, on co-ordinate rings,
$$\phi^*:k[\overline{y}, \{w_{\beta,i}: 0\leq i\leq\ell-1, \beta\in\mathbb{N}^r,0<|\beta|\leq m\}]\to k[\overline{y},\{z_\gamma:\gamma\in\mathbb{N}^{\ell r},0<|\gamma|\leq m\} ]$$
is given by
\begin{itemize}
\item
$\overline{y}\longmapsto\overline{y}$
\item
$w_{\beta,0}\longmapsto z_{\hat\beta}$ for each $\beta\in\mathbb{N}^r,0<|\beta|\leq m$
\item
$\displaystyle w_{\beta,j}\longmapsto \big(\sum_{\gamma\in\Gamma_\beta}c_{\gamma,j}\big)z_\gamma$
for each $\beta$ and $j=1,\dots,\ell-1$.
\qed
\end{itemize}
\end{proposition}

The following corollary of the above co-ordinate description will be useful in a sequel to this article where we study ``$\E$-schemes" and their ``$\E$-jets".

\begin{corollary}
\label{phiprops}
Suppose $k$ is a field and $X$ is of finite type over $k$.
If $p\in X(k)$ is smooth then $\phi$ restricts to a surjective linear map between the fibres of $\jet^m\tau(X)$ and $\tau\jet^m(X)$ over $\nabla(p)\in\tau(X)(k)$.
\end{corollary}

\begin{proof}
By Lemma~\ref{interpislin}, $\phi$ is a morphism of linear spaces over $\tau(X)$, and so all that requires proof is the surjectivity.

By smoothness at $p$, there exists an \'etale map from a nonempty smooth Zariski open subset $U\subseteq X$ containing $p$ to $\mathbb{A}_k^r$ for some $r\geq 0$.
Note that $\nabla(p)\in\tau(U)$ and $\jet^m(\tau U)=\jet^m(\tau X)|_{\tau U}$.
Hence $\jet^m(\tau U)_{\nabla(p)}=\jet^m(\tau X)_{\nabla(p)}$.
Also, since $\jet^m(U)=\jet^m(X)|_U$, we have
$$\big(\tau\jet^m(U)\big)_{\nabla(p)}=\tau(\jet^m U_p)=\tau(\jet^m X_p)=\big(\tau\jet^m(X)\big)_{\nabla(p)}$$
where the first and final equalities are by Proposition~\ref{nablafunctorial}(b).
So, without loss of generality, we may replace $X$ by $U$, and assume there is an \'etale map $f:X\to\mathbb{A}_k^r$.
Indeed, under this hypothesis, we will show that for any $p'\in\tau(X)_p(k^{\alg})$, $\phi_{p'}$ is a surjection from $\jet^m(\tau X)_{p'}$ to $(\tau \jet^mX)_{p'}$.
For $X=\mathbb{A}_k^r$ this is Proposition~\ref{affinesurjective}.

By Proposition~\ref{etaletau}, $\tau(f)$ is \'{e}tale, and so by Lemma~\ref{etalejet}, $\jet^m\big(\tau(X)\big)_{p'}$ is isomorphic to $\jet^m\big(\tau(\mathbb{A}_k^r)\big)_{\tau(f)(p')}$.
It remains therefore to prove that $\tau\big(\jet^m(f)\big)$ induces an isomorphism between the fibres over $p'$ and $\tau(f)(p')$ in the following diagram
$$\xymatrix{
\tau\big(\jet^m(X)\big)\ar[rr]^{\tau\big(\jet^m(f)\big)}\ar[d] && \tau\big(\jet^m(\mathbb{A}_k^r)\big)\ar[d]\\
\tau(X)\ar[rr]^{\tau(f)} & &\tau(\mathbb{A}_k^r)
}$$
This in turn reduces to showing that
if $\hat p$ is the $\E^e(k^{\alg})$-point of $X$ corresponding to $p'$, then in the following diagram
$$\xymatrix{
\jet^m(X)\ar[rr]^{\jet^m(f)}\ar[d] & & \jet^m(\mathbb{A}_k^r)\ar[d]\\
X\ar[rr]^{f} && \mathbb{A}_k^r
}$$
$\jet^{m}(f)$ induces an isomorphism between the fibres over $\hat p$ and $f(\hat p)$.
But as $f$ is \'etale, this is just Lemma~\ref{etalejet} with $R=\E^e(k^{\alg})$ apllied to $\hat{p}$.
\end{proof}

\bigskip
\bigskip


\end{document}